\documentclass[12pt]{amsart}
\usepackage{amsmath,amssymb,mathrsfs,hyperref,color}

\usepackage[nameinlink]{cleveref}
\hypersetup{colorlinks={true},linkcolor={blue},citecolor=green}
\crefformat{equation}{(#2#1#3)}

\usepackage{mathtools}
\usepackage[x11names]{xcolor}
\newtagform{blue}{\color{blue}(}{)}

\usepackage{paralist}

\usepackage{subfigure}
\usepackage{float}
\usepackage{times}
\usepackage{tikz}
\usepackage{cases}
\usetikzlibrary{intersections}
\usepackage{xcolor}
\usepackage[latin1]{inputenc}
\usepackage{bbm}
\usetikzlibrary{shadings}
\usetikzlibrary[intersections]
\usetikzlibrary{arrows,shapes}

\makeatletter
\def\setliststart#1{\setcounter{\@listctr}{#1}%
  \addtocounter{\@listctr}{-1}}
\makeatother

\makeatletter
\@addtoreset{figure}{section}
\makeatother

\usepackage{calc}
\newtheorem{theorem}{Theorem}[section]
\newtheorem{lemma}[theorem]{Lemma}
\newtheorem{proposition}[theorem]{Proposition}
\newtheorem{corollary}[theorem]{Corollary}
\newtheorem{remarks}[theorem]{Remark}

\newtheorem{definition}[theorem]{Definition}
\numberwithin{equation}{section}

\newcommand{\R}{\mathbb{R}}

\newcommand{\N}{\mathbb{N}}

\newcommand{\K}{\mathcal{K}}
\newcommand{\PP}{\mathcal{P}}
\newcommand{\A}{\mathcal{A}}
\newcommand{\C}{\mathcal{C}}

\newcommand{\LL}{\mathcal{L}}

\DeclareMathOperator*{\supp}{spt}
\DeclareMathOperator*{\esssup}{ess\ sup}

\DeclareMathOperator*{\eps}{\varepsilon}

\DeclareMathOperator*{\PB}{A_{\infty}}
\DeclareMathOperator*{\mane}{\alpha({\it L})}
\DeclareMathOperator*{\SR}{SR}
\DeclareMathOperator*{\Cal}{Cal}

\makeatletter
\def\moverlay{\mathpalette\mov@rlay}
\def\mov@rlay#1#2{\leavevmode\vtop{%
   \baselineskip\z@skip \lineskiplimit-\maxdimen
   \ialign{\hfil$\m@th#1##$\hfil\cr#2\crcr}}}
\newcommand{\charfusion}[3][\mathord]{
    #1{\ifx#1\mathop\vphantom{#2}\fi
        \mathpalette\mov@rlay{#2\cr#3}
      }
    \ifx#1\mathop\expandafter\displaylimits\fi}
\makeatother

\title[Aubry set for sub-Riemannian control systems]{Aubry set for sub-Riemannian control systems}
\author{Piermarco Cannarsa \and Cristian Mendico}
\address{Dipartimento di Matematica, Universit\`a di Roma ``Tor Vergata'', Via della Ricerca Scientifica 1, 00133 Roma, Italy}
\email{cannarsa@mat.uniroma2.it}
\address{Dipartimento di Matematica, Universit\`a di Roma ``Tor Vergata'', Via della Ricerca Scientifica 1, 00133 Roma, Italy}
\email{mendico@axp.mat.uniroma2.it}

\date{\today}

\begin{document}
\usetagform{blue}
\begin{abstract} 
In the paper [P. Cannarsa, C. Mendico, {\it Asymptotic analysis for Hamilton-Jacobi-Bellman equations on Euclidean space}, (2021) Arxiv], we proved the existence of the limit as the time horizon goes to infinity of the averaged value function of an optimal control problem. For the classical Tonelli case such a limit is called the critical constant of the problem. In the special case of sub-Riemannian control systems, we also proved the existence of a critical solution, that is, a continuous solution to the Hamilton-Jacobi equation associated with such a constant, which also coincides with its Lax-Oleinik evolution. Here, we focus  our attention on the sub-Riemannian case providing a variational representation formula for the critical constant which uses an adapted notion of closed measures. Having such a formula at our disposal, we define and study the Aubry set. First, we investigate dynamical and topological properties of such a set w.r.t. a suitable class of minimizing trajectories of the Lagrangian action. Then, we show that critical solutions to the Hamilton-Jacobi equation are horizontally differentiable  and satisfy the equation in classical sense on the Aubry set.
\end{abstract}
\maketitle

\section{Introduction}
In the recent paper \cite{bib:CM02}, we have studied the asymptotic behaviour as $T\to+\infty$ of the value function
\begin{equation}\label{eq:valuepre}
V_T(x)= \inf_{u}\int_{0}^{T}{L(\gamma^{x}_{u}(t), u(t))\ dt} 
\qquad(x\in\R^d)
\end{equation}
where $L$ is a Tonelli Lagrangian,  $u:[0,T]\to\R^m\; (m \leq d)$ is a square integrable control function,  and $\gamma^x_u$ is the solution of the state equation
\begin{equation}\label{eq:preldyn}
\begin{cases}
 \dot\gamma(t)= \sum_{i=1}^{m}{u_{i}(t)f_{i}(\gamma(t))}
 &
 \mbox{ a.e.}\,\, t\in [0,T]
 \vspace{.1cm}
 \\
 \gamma(0)=x\,,
\end{cases}	
\end{equation}
where $\{f_1,\dots,f_m\}$ are linearly independent smooth vector fields satisfying the so-called Lie algebra rank condition. Observe that such an assumption ensures that system \eqref{eq:preldyn} is small time locally controllable. 
By using such a property and assuming the existence of a compact attractor for the  optimal trajectories of \eqref{eq:valuepre}, we proved in \cite{bib:CM02} that
\begin{equation}\label{eq:prelmane}
\lim_{T \to \infty} \frac{1}{T}V_T(x)=\mane
\end{equation}
uniformly on all bounded subsets of $\R^d$. Hereafter, we call $\mane$ the {\it critical constant} of $L$. In fact, the existence of the limit in \eqref{eq:prelmane} was obtained in \cite{bib:CM02} for a general control system of the form 
\begin{equation*}
\dot\gamma(t)=g(\gamma(t), u(t)) 
\end{equation*}
under a suitable controllability condition.

As is well known (see \cite{bib:MA} and the Appendix of \cite{bib:CM02}), the convergence of time averages \eqref{eq:prelmane} entails the (locally uniform) convergence as $\lambda\downarrow 0$ of Abel means $\{\lambda v_\lambda\}_{\lambda>0}$, where
$$
v_{\lambda}(x)=\inf_{u(\cdot)} \int_{0}^{\infty}{e^{-\lambda t}L(\gamma^{x}_{u}(s), u(s))\ ds}\qquad (x\in \R^d)\,.
$$
Using such a result, in \cite{bib:CM02}, we were able to construct a continuous viscosity solution, $\chi$, of the so-called critical Hamilton-Jacobi equation
\begin{equation}\label{eq:criticalprel}
\mane + H(x, D\chi(x))=0\qquad (x\in \R^d)\,,
\end{equation}
where $H$ is defined by
\begin{equation}\label{eq:intro_H}
H(x,p)=\sup_{u \in \R^{m}}\left\{\sum_{i=1}^{m}{u_{i}} \langle p_{i},f_{i}(x) \rangle  - L(x,u)\right\}, 
\qquad \forall\ (x,p)\in\R^d\times \R^d.
\end{equation}
Moreover, in  \cite{bib:CM02}, $\chi$ was proved to satisfy the representation formula
\begin{equation}\label{eq:preLA}
\chi(x)=\inf_{(\gamma, u) :\ \gamma(t)=x} \left\{\chi(\gamma(0)) + \int_{0}^{t}{L(\gamma(s), u(s))\ ds} \right\} - \mane t. 
\end{equation}
Hereafter, any solution $\chi$ to \eqref{eq:criticalprel} satisfying \eqref{eq:preLA} will be called a {\it critical solution} for equation \eqref{eq:criticalprel}. Notice that, given a critical solution $\chi$ for \eqref{eq:criticalprel} and any constant $c \in \R$, we have that $\chi + c$ is again a critical solution for \eqref{eq:criticalprel}.

In the classical Tonelli case---where both $L(x,v)$ and $H(x,p)$ are smooth functions, strictly convex and superlinear in $v$ and $p$, respectively---it is known that the critical constant has a powerful interpretation in terms of a suitable class of probability measures minimizing the Lagrangian action on the tangent bundle. This connection is well explained by the celebrated Aubry-Mather and weak KAM theories  (see, for instance, \cite{bib:FA, bib:SOR, bib:MAT} and the references therein). On the other hand, the Hamiltonian in \eqref{eq:intro_H} fails to be Tonelli, in general. So,  classical weak KAM theory does not apply to minimization problems associated with sub-Riemannian control systems. The purpose of this paper is to develop new methods to make such an extension possible. 

Our approach is based on  metric properties of the geometry induced on the state space by the sub-Riemannian structure. In particular, the existence of geodesics w.r.t. the sub-Riemannian metric plays a crucial role in our study of the regularity properties of solutions to \eqref{eq:criticalprel}. We refer the reader to  \cite{bib:ABB, bib:COR, bib:LR} and the references therein for more details on sub-Riemannian geometry. 

Moreover, in order to improve the natural regularity of critical solutions---which, in general, are just H\"older continuous, see \cite{bib:CM02}---we restrict the analysis to the class of sub-Riemannian systems that admit no singular minimizing controls different from zero. We recall that a control function is singular for the problem if there exists a dual arc which is orthogonal to vector fields $\{f_{i}\}_{i=1,\dots, m}$ along the associated trajectory (\Cref{thm:singulararc}). When no singular minimizing controls are present, owing to \cite{bib:CR} we have that critical solutions are locally semiconcave, hence locally Lipschitz, on $\R^d$. 

In order to deal with unbounded state and control spaces, we assume the existence of a compact set $\mathcal{K}_{L} \subset \R^{d}$ such that
	\begin{equation*}
		 \inf_{x \in \R^{d} \backslash \mathcal{K}_{L}} L(x,0) > \min_{x \in \mathcal{K}_{L}} L(x,0). 
	\end{equation*}
	 Notice that $\K_{L}$ acts as a compact attractor for all optimal trajectories: assuming the existence of such an attractor is customary in this kind of problems (see \cite{bib:CCMW}).

Finally, we recall that a crucial role in the classical Aubry-Mather Theory is played by probability measures on the tangent bundle which are invariant for the Euler-Lagrange flow. In sub-Riemannian settings this strategy fails because such a flow does not exist. For this reason, we replace invariant measures by a suitable notion of closed measures, tailored to fit sub-Riemannian structures. Closed measures were first introduced in \cite{bib:FAS} in order to adapt weak KAM theory to problems in the calculus of variations satisfying mild regularity assumptions. In our case, setting
\begin{equation*}
\PP_{c}^{2}(\R^{d} \times \R^{m}) = \left\{\mu \in \PP(\R^{d} \times \R^{m}) : \int_{\R^{d} \times \R^{m}}{|u|^{2}\ \mu(dx,du)} < +\infty,\ \supp(\pi_{1} \sharp \mu)\, \text{compact} \right\}	
\end{equation*}
where $\pi_{1}: \R^{d} \times \R^{m} \to \R^{d}$ denotes projection onto the first factor, i.e. $\pi_{1}(x,u)= x$, we define $\mu \in \PP_{c}^{2}(\R^{d} \times \R^{m})$ to be an $F$-closed measure if
	\begin{equation*}
	\int_{\R^{d} \times \R^{m}}{\langle F^{\star}(x)D\varphi(x), u \rangle\ \mu(dx,du)}=0, \quad \forall\ \varphi \in C^{1}(\R^{d}). 	
	\end{equation*}
	 Note that closed measures are naturally supported by the distribution associated with $\{f_{i}\}_{i=1,\dots,m}$ which, in our case, reduces to $\R^d\times\R^m$.

Denoting by $\C_{F}$ the set of all $F$-closed probability measures, we show that (\Cref{thm:main1}) 
\begin{equation*}
\mane = \inf_{\mu \in \C_{F}} \int_{\R^{d} \times \R^{m}}{L(x,u)\ \mu(dx,du)}.
\end{equation*} 
Then, we consider the action functional, $A_{t}(x,y)$, defined by
		\begin{equation*}
		A_{t}(x,y)=\inf_{(\gamma, u) : \substack{\gamma(0)=x \\ \gamma(t)=y}} \left\{\int_{0}^{t}{L(\gamma(s), u(s))\ ds} \right\},
		\end{equation*}
		and Peierls barrier 
\begin{align*}
\PB(x,y) =  \liminf_{t \to \infty}\ \big[A_{t}(x,y) - \mane t\big] \quad (x,y \in \R^{d}).	
\end{align*}
In the sub-Riemannian context, we derive several properties of Peierls barrier which are known for Hamilton-Jacobi equations of Tonelli type on compact manifolds. In particular, we show that $\PB(x,\cdot)$ is a critical solution of
\eqref{eq:criticalprel} (Corollary~\ref{cor:viscoHJ}). Finally, we study the projected Aubry set,  $\A$, given by
\begin{equation*}
\A := \{x \in \R^{d}: \PB(x,x)=0\}. 
\end{equation*}
After showing that $\A$ is a compact invariant subset of $\R^d$  for the class of calibrated curves for \eqref{eq:preLA}   (\Cref{thm:compactness} and \Cref{prop:aubrycal}), we address the issue of the horizontal differentiability of critical solutions on $\A$. We recall that a continuous function $\psi$ on $\R^{d}$ is differentiable at $x \in \R^{d}$ along the range of $F(x)$ (or, horizontally differentiable at $x$) if there exists $q_{x} \in \R^{m}$ such that
\begin{equation*}
\lim_{v \to 0} \frac{\psi(x + F(x)v) - \psi(x) - \langle q_{x}, v \rangle}{|v|} = 0.	
\end{equation*}
The vector $q_{x}$, also denoted by $D_{F}\psi(x)$, is called the horizontal gradient of $\psi$. By appealing to  the semiconcavity of critical solutions---which is ensured  by \cite{bib:CR} as soon as  we assume the absence of singular minimising controls---we prove that these solutions are differentiable along the range of vector fields $\{f_{i}\}_{i=1,\cdots,m}$ at any point $x \in \A$ (see \Cref{thm:aubryset}). Finally we analyze the behavior of the horizontal gradient of any critical solution showing that for any $x \in \A$
\begin{equation*}
D_{F}\chi(\gamma_{x}(t))=D_{u}L(\gamma_{x}(t), u_{x}(t)) \quad (t \in\R)
\end{equation*}
along any calibrated pair $(\gamma_{x}, u_{x})$ for $\chi$ with $\gamma_{x}(0)=x$. 

\medskip

The rest of the paper is organized as follows. In \Cref{sec:preliminaries}, we fix the notation used throughout the paper and recall some preliminaries. In \Cref{sec:Settingsassumptions}, we introduce the problem and we describe the assumptions we make in this paper. In \Cref{sec:ergodicconstant}, we prove the aforementioned characterization of the critical constant $\mane$. In \Cref{sec:aubrymather}, we define Peierls barrier and describe several properties of such a function. Then, we define the projected Aubry set showing that it is a compact subset of $\R^{d}$ invariant for a set of suitable action minimizing trajectories. Finally, in \Cref{sec:horizontal} we prove that any critical solution is horizontally differentiable on the projected Aubry set and we provide a representation formula for the horizontal gradient.


\medskip
{\bf Acknowledgements.} Piermarco Cannarsa was partly supported by Istituto Nazionale di Alta Matematica (GNAMPA 2020 Research Projects) and by the MIUR Excellence Department Project awarded to the Department of Mathematics, University of Rome Tor Vergata, CUP E83C18000100006. Cristian Mendico was partly supported by Istituto Nazionale di Alta Matematica (GNAMPA 2020 Research Projects) and University Italo Francese (Vinci Project 2020) n. C2-794. Part of this paper was completed while the second author was visiting in Department of Mathematics of the University of Rome Tor Vergata.

\medskip
{\bf Conflict of interest.} The authors declare that there is no conflict of interest.

\section{Preliminaries}
\label{sec:preliminaries}
\subsection{Notation}
We write below a list of symbols used throughout this paper.
\begin{itemize}
	\item Denote by $\mathbb{N}$ the set of positive integers, by $\mathbb{R}^d$ the $d$-dimensional real Euclidean space,  by $\langle\cdot,\cdot\rangle$ the Euclidean scalar product, by $|\cdot|$ the usual norm in $\mathbb{R}^d$, and by $B_{R}$ the open ball with center $0$ and radius $R$.
	\item If $\Lambda$ is a real $n\times n$ matrix, we define the norm of $\Lambda$ by 
\[
\|\Lambda\|=\sup_{|x|=1,\ x\in\mathbb{R}^d}|\Lambda x|.
\]

\item For a Lebesgue-measurable subset $A$ of $\mathbb{R}^d$, we let $\mathcal{L}^{d}(A)$ be the $d$-dimensional Lebesgue measure of $A$ and  $\mathbf{1}_{A}:\mathbb{R}^n\rightarrow \{0,1\}$ be the characteristic function of $A$, i.e.,
\begin{align*}
\mathbf{1}_{A}(x)=
\begin{cases}
1  \ \ \ &x\in A,\\
0 &x \not\in A.
\end{cases}
\end{align*} 
We denote by $L^p(A)$ (for $1\leq p\leq \infty$) the space of Lebesgue-measurable functions $f$ with $\|f\|_{p,A}<\infty$, where   
\begin{align*}
& \|f\|_{\infty, A}:=\esssup_{x \in A} |f(x)|,
\\& \|f\|_{p,A}:=\left(\int_{A}|f|^{p}\ dx\right)^{\frac{1}{p}}, \quad 1\leq p<\infty.
\end{align*}
For brevity, $\|f\|_{\infty}$  and $\|f\|_{p}$ stand for  $\|f\|_{\infty,\mathbb{R}^d}$ and  $\|f\|_{p,\mathbb{R}^d}$ respectively.

\item $C_b(\mathbb{R}^d)$ stands for the function space of bounded uniformly  continuous functions on $\mathbb{R}^d$, $C^{2}_{b}(\mathbb{R}^{d})$ for the space of bounded functions on $\mathbb{R}^d$ with bounded uniformly continuous first and second derivatives, and $C^k(\mathbb{R}^{d})$ ($k\in\mathbb{N}$) for the function space of $k$-times continuously differentiable functions on $\mathbb{R}^d$. We set $C^\infty(\mathbb{R}^{d}):=\cap_{k=0}^\infty C^k(\mathbb{R}^{d})$ and we denote by $C_c^\infty(\mathbb{R}^{d})$ the space of functions in $C^\infty(\mathbb{R}^{d})$ with compact support. 

\item Let $a<b\in\mathbb{R}$. $AC([a,b];\mathbb{R}^d)$ denotes the space of absolutely continuous maps $[a,b]\to \mathbb{R}^d$.
  
  \item For $f \in C^{1}(\mathbb{R}^{d})$, the gradient of $f$ is denoted by $Df=(D_{x_{1}}f, ..., D_{x_{n}}f)$, where $D_{x_{i}}f=\frac{\partial f}{\partial x_{i}}$, $i=1,2,\cdots,d$.
Let $k$ be a nonnegative integer and let $\alpha=(\alpha_1,\cdots,\alpha_d)$ be a multiindex of order $k$, i.e., $k=|\alpha|=\alpha_1+\cdots +\alpha_d$ , where each component $\alpha_i$ is a nonnegative integer.   For $f \in C^{k}(\mathbb{R}^{d})$,
define $D^{\alpha}f:= D_{x_{1}}^{\alpha_{1}} \cdot\cdot\cdot D^{\alpha_{d}}_{x_{d}}f$.
\item For any $\eps > 0$ we say that $\xi_{\eps} \in C^{\infty}(\R^{d})$ is a smooth mollifier if
\begin{equation*}
\supp(\xi^{\eps}) \subset B_{\eps}, \quad 0 \leq \xi^{\eps}(x) \leq 1 \,\, \forall\ x \in \R^{d}, \quad \int_{B_{\eps}}{\xi^{\eps}(x)\ dx} = 1.	
\end{equation*}
\item Given a continuous function $\varphi : \R^{d} \to \R$ we denote by $D^{+}\varphi(x)$ the superdifferential of $\varphi$ at $x$, that is,
\begin{equation*}
D^{+}\varphi(x)=\left\{p \in \R^{d}: \limsup_{y \to x} \frac{\varphi(y) - \varphi(x) - \langle p, x-y \rangle}{|x-y|} \leq 0 \right\}.
\end{equation*}
\item For any $R \geq 0$ we say that $\xi_{R}$ is a smooth cut-off function if $\xi_{R}(x)={\bf 1}_{\overline{B}_{R}} \star \xi_{\frac{1}{2}}(x)$ where $\star$ denotes the convolution operator. 
\end{itemize}

\subsection{Measure theory}
We recall here some preliminary material concerning Wasserstein spaces and Wasserstein distance. For more details we refer the reader to \cite{bib:CV, bib:AGS}. 

Let $(X,{\bf d})$ be a metric space (in the paper, we use $X= \R^d$ or $X= \R^d\times \R^m$). 
Denote by $\mathcal{B}(X)$ the  Borel $\sigma$-algebra on $X$ and by $\mathcal{P}(X)$ the space of Borel probability measures on $X$.
The support of a measure $\mu \in \mathcal{P}(X)$, hereafter denoted by $\supp(\mu)$, is the closed set defined by
\begin{equation*}
\supp (\mu) := \Big \{x \in X: \mu(V_x)>0\ \text{for each open neighborhood $V_x$ of $x$}\Big\}.
\end{equation*}
We say that a sequence $\{\mu_k\}_{k\in\mathbb{N}}\subset \mathcal{P}(X)$ is weakly-$*$ convergent to $\mu \in \mathcal{P}(X)$, and we write
$\mu_k \stackrel{w^*}{\longrightarrow}\mu$,
  if
\begin{equation*}
\lim_{n\rightarrow \infty} \int_{X} f(x)\,d\mu_n(x)=\int_{X} f(x) \,d\mu(x), \quad  \forall f \in C_b(X).
\end{equation*}
There exists an interesting link between weak-$*$ convergence and convergence of the support of measures, see \cite[Proposition 5.1.8]{bib:AGS}. Indeed, if $\{\mu_{j}\}_{j \in \N} \subset \PP(X)$ weakly-$*$ converges to $\mu \in \PP(X)$, then
\begin{equation}\label{eq:kuratoski}
\forall\ x \in \supp(\mu) \quad \exists x_{j} \in \supp(\mu_{j}) \quad \text{such that} \quad \lim_{j \to \infty} x_{j}=x.
\end{equation}

For $\alpha \in[1,+\infty)$, the Wasserstein space of order $\alpha$ is defined as
\begin{equation*}
\mathcal{P}^{\alpha}(X):=\left\{m\in\mathcal{P}(X): \int_{X} d(x_0,x)^{\alpha}\,dm(x) <+\infty\right\},
\end{equation*}
for some (thus, all) $x_0 \in X$.
Given any two measures $m$ and $m^{\prime}$ in $\mathcal{P}^{\alpha}(X)$, we define
\begin{equation}\label{def.transportplan}
\Pi(m,m'):=\Big\{\lambda\in\mathcal{P}(X \times X): \lambda(A\times X)=m(A),\ \lambda(X \times A)=m'(A),\ \forall A\in \mathcal{B}(X)\Big\}.
\end{equation}
The Wasserstein distance of order $\alpha$ between $m$ and $m'$ is defined by
    \begin{equation*}\label{dis1}
          d_{\alpha}(m,m')=\inf_{\lambda \in\Pi(m,m')}\left(\int_{X\times X}d(x,y)^{\alpha}\,d\lambda(x,y) \right)^{1/{\alpha}} \quad (m, m' \in \PP^{\alpha}(X)).
    \end{equation*}
    The distance $d_1$ is also commonly called the Kantorovich-Rubinstein distance and can be characterized by a useful duality formula (see, for instance, \cite{bib:CV})  as follows
\begin{equation}\label{eq:2-100}
d_1(m,m')=\sup\left\{\int_{X} f(x)\,dm(x)-\int_{X} f(x)\,dm'(x) \ |\ f:X\rightarrow\mathbb{R} \ \ \text{is}\ 1\text{-Lipschitz}\right\},
\end{equation}
for all $m$, $m'\in\mathcal{P}^{1}(X)$.

Let $X_{1}$, $X_{2}$ be metric spaces, let $\mu \in \PP(X_{1})$, and let $f: X_{1} \to X_{2}$ be a $\mu$-measurable map. We denote by $f \sharp \mu \in \PP(X_{2})$ the push-forward of $\mu$ through $f$ defined as
\begin{equation*}
f \sharp \mu(B):= \mu(f^{-1}(B)), \quad \forall\ B \in \mathcal{B}(X_{2}).	
\end{equation*}
Observe that, in integral form, the push-forward operates on $C_{b}(X_{2})$ as follows
\begin{equation*}
\int_{X_{1}}{\varphi(f(x))\ \mu(dx)} = \int_{X_{2}}{\varphi(y)\ f \sharp \mu(dy)} \quad \forall\ \varphi \in C_{b}(X_{2}).	
\end{equation*}

\subsection{Nonholonomic control systems}
We now recall the notion of nonholonomic control systems and sub-Riemannian distance on $\R^{d}$, see \cite{bib:MON, bib:LR, bib:ABB}.

A nonholonomic system on $\R^{d}$ is a control system of the form
\begin{equation}\label{eq:preliminary}
	\dot\gamma(t)=\sum_{i=1}^{m}{u_{i}(t) f_{i}(\gamma(t))}, \quad t \in [0,+\infty), \,\, u_{i}(t) \in \R, \, i = 1,\dots,m
\end{equation}
where $f_{i}$ are linearly independent smooth vector fields. 
Such a system induces a distance on $\R^{d}$ in the following way. Define the sub-Riemannian metric to be the function $g: \R^{d} \times \R^{m} \to \R \cup \{ \infty\}$ given by
\begin{equation*}
g(x,v)=\inf\left\{\sum_{i =1}^{m}{u_{i}^{2}: v=\sum_{i=1}^{m}{u_{i} f_{i}(x)}}\right\}.	
\end{equation*}
If $v \in \text{span}\{f_{1}(x), \dots, f_{m}(x)\}$ then the infimum is attained at a unique value $u_{x} \in \R^{m}$ and $g(x,v)=| u_{x}|^{2}$. Then, since $g(\gamma(t), \dot\gamma(t))$ is measurable, being the composition of the lower semicontinuous function $g$ with a measurable function, we can define the length of an absolutely continuous curve $\gamma:[0,T] \to \R^{d}$ as
\begin{equation*}
\text{length}(\gamma)=\int_{0}^{T}{\sqrt{g(\gamma(t), \dot\gamma(t))}\ dt}.	
\end{equation*}
Note that since $\sqrt{g(x,v)}$ is homogenous of degree $1$ with respect to the second variable we obtain that the length of any absolutely continuous curve is independent of $T$.
 In conclusion, one defines the sub-Riemannian distance as
 \begin{equation*}
 d_{\text{SR}}(x,y)=\inf_{(\gamma, u) \in \Gamma_{0,T}^{x \to y}}\text{length}(\gamma)	
 \end{equation*}
where $\Gamma_{0,T}^{x \to y}$ denotes the set of all trajectory-control pairs such that $u \in L^{2}(0,T; \R^{m})$, $\gamma$ solves \eqref{eq:preldyn} for such a control $u$, $\gamma(0)=x$ and $\gamma(T)=y$,

Following \cite{bib:CR} it is possible to characterize the sub-Riemannian distance as follows
\begin{equation}\label{eq:subriem}
	d_{\SR}(x,y) = \inf\left\{T >0: \exists\ (\gamma, u) \in \Gamma_{0,T}^{x \to y},\ |u(t)| \leq 1\ \text{a.e.}\ t \in [0,T] \right\}
\end{equation}
for any $x$, $y \in \R^{d}$. 

\begin{definition}
Given $x$, $y \in \R^{d}$ we say that $(\gamma, u)$ is a geodesic pair for $x$ and $y$ if $(\gamma, u) \in \Gamma_{0, d_{\SR}(x,y)}^{x \to y}$ and $|u(t)| \leq 1$ for a.e. $t \in [0,d_{\SR}(x,y)]$.  
\end{definition}

We conclude this preliminary part with a brief introduction to singular controls. Let $x_{0} \in \R^d$ and fix $t >0$. The end-point map associated with system \eqref{eq:preliminary} is the function  
\begin{equation*}
E^{x_{0},t}: L^{2}(0,t; \R^{m}) \to \R^{d}	
\end{equation*}
defined as
\begin{equation*}
E^{x_{0},t}(u)=\gamma(t)	
\end{equation*}
where $\gamma$ is the solution of \eqref{eq:preliminary} associated with $u$ such that $\gamma(0)=x_{0}$. 
 Under the assumption that the vector field $f_{i}$ has sub-linear growth, for any $i=1, \dots m$ it is known that $E^{x_{0},t}$ is of class $C^{1}$ on $L^{2}(0,t;\R^{m})$. Then, we say that a control $\bar{u} \in L^{2}(0,t;\R^{m})$ is singular for $E^{x_{0},t}$ if $dE^{x_{0},t}(\bar{u})$ is not surjective. Moreover, defining the function $H_{0}: \R^{d} \times \R^{d} \times \R^{m} \to \R$ as
\begin{equation*}
	H_{0}(x,p,u)=\sum_{i=1}^{m}{u_{i}\langle p,f_{i}(x) \rangle},
\end{equation*}
we have the following well-known characterization of singular controls (e.g., \cite{bib:CR}). 
\begin{theorem}\label{thm:singulararc}
A control $u \in L^{2}(0,t; \R^{m})$ is singular for $E^{x_{0},t}$ if and only if there exists an absolutely continuous arc $p:[0,t] \to \R^{d}\backslash\{0\}$ such that
\begin{align*}
\begin{cases}
	\dot\gamma(s) & =\  D_{p}H_{0}(\gamma(s), p(s), u(s))
	\\
	-\dot{p}(s) & =\   D_{x}H_{0}(\gamma(s), p(s), u(s))
\end{cases}	
\end{align*}
	with $\gamma(0)=x_{0}$ and 
	\begin{equation*}
	D_{u}H_{0}(\gamma(s), p(s),u(s))=0, \quad \text{for a.e.}\ s \in [0,t],
	\end{equation*}
	that is, 
	\begin{equation*}
		\langle f_{i}(\gamma(s)), p(s) \rangle =0 \quad (i=1,\dots,m)
	\end{equation*}
for any $s \in [0,t]$.   
\end{theorem}

\section{Settings and assumptions}
\label{sec:Settingsassumptions}
	
	Fixed $m \in \N$ such that $m \leq d$, let 
	\begin{equation*}
		f_{i}: \R^{d} \to \R^{d} \quad (i=1, \dots, m)
	\end{equation*}
	and
	\begin{equation*}
		u_{i}:[0,\infty) \to \R \quad (i=1, \dots, m)
	\end{equation*}
be smooth vector fields and square-integrable controls, respectively. Consider the following controlled dynamics of sub-Riemannian type
	\begin{align}\label{eq:dynamics}
         \dot\gamma(t)= \displaystyle{\sum_{i=1}^{m}{u_{i}(t) f_{i}(\gamma(t))}}= F(\gamma(t))U(t), \quad t \in [0,+\infty)
	\end{align}
	where $F(x)=[f_{1}(x)| \dots | f_{m}(x)]$ is a $d \times m$ real matrix and $U(t)=(u_{1}(t), \dots, u_{m}(t))^{\star}$\footnote{$(u_{1}, \dots, u_{m})^{\star}$ denotes the transpose of $(u_{1}, \dots, u_{m})$}. 

	For any $s_{0}$, $s_{1} \in \R$ with $s_{0} < s_{1}$ and $x$, $y \in \R^{d}$ we set
	\begin{align*}\label{eq:controlnotation}
	\begin{split}
	\Gamma_{s_{0}, s_{1}}^{x \to} & =\{(\gamma, u) \in \text{AC}([s_{0}, s_{1}]; \R^{d}) \times L^{2}(s_{0}, s_{1}; \R^{m}): \dot\gamma(t)=F(\gamma(t))u(t),\,\, \gamma(s_{0})=x\},
	\\
	 \Gamma_{s_{0}, s_{1}}^{\to y} & =\{(\gamma, u) \in \text{AC}([s_{0}, s_{1}]; \R^{d}) \times L^{2}(s_{0}, s_{1}; \R^{m}): \dot\gamma(t)=F(\gamma(t))u(t),\,\, \gamma(s_{1})=y\},
	 \\
	 \Gamma_{s_{0}, s_{1}}^{x \to y} & = \Gamma_{s_{0}, s_{1}}^{x \to} \cap \Gamma_{s_{0}, s_{1}}^{\to y}.
	 \end{split}
	\end{align*}

Throughout this paper we assume vector fields $f_{i}$ to satisfy the following.
\begin{itemize}
\item[{\bf (F1)}] There exists a constant $c_{f} \geq 1$ such that for any $i=1, \dots, m$
 \begin{equation}\label{eq:fassum}
	|f_{i}(x)| \leq c_{f}(1+|x|), \quad \forall\ x \in \R^{d}.
\end{equation}
	\item[{\bf (F2)}] $F \in C^{1,1}_{loc}(\R^{d})$ and for any $x \in \R^{d}$ the matrix $F(x)$ has full rank (equal to $m$).
\end{itemize}
Notice by, from {\bf (F2)} we have that vector fields $f_{i}$ are linearly independent. Moreover, by {\bf (F1)} and Gronwall's inequality we get the following estimate for solutions of \eqref{eq:dynamics}. 
\begin{lemma}\label{lem:boundedtrajectories}
Let $x \in \R^{d}$, $t \geq 0$ and $(\gamma, u) \in \Gamma_{0,t}^{x \to}$. If $u \in L^{\infty}(0,t; \R^{m})$ then we have that
	\begin{equation*}
	|\gamma(s)| \leq (|x| + c_{f}\|u\|_{\infty}t)e^{c_{f}\|u\|_{\infty}t}, \quad \forall\ s \in [0,t].	
	\end{equation*}
	\end{lemma}

We now state the assumptions on the Lagrangian $L: \R^{d} \times \R^{m} \to \R$.
\begin{itemize}
\item[({\bf L1})] There exists a constant $C_{0} \geq 0$ such that 
\begin{equation*}
L(0,u) \leq C_{0}(1+|u|^{2}), \quad \forall\ u \in \R^{m}
\end{equation*}
and $L(x,u)=L(x,-u)$ for any $(x,u) \in \R^{d} \times \R^{m}$. 
\item[{\bf (L2)}] $L \in C^{2}(\R^{d} \times \R^{m})$  and there exist positive constants $\ell_{1}$, $C_1$ such that
\begin{align*}
|D_{x}L(x,u)| \leq\  C_{1}(1+|u|^{2}), & \quad (x,u) \in \R^{d} \times \R^{m},
\\
D^{2}_{u}L(x,u) \geq\  \frac{1}{\ell_{1}}, & \quad (x,u) \in \R^{d} \times \R^{m},
\\
D^{2}_{x}L(x,u) \leq\  \ell_{1}, & \quad (x,u) \in \R^{d} \times \R^{m}.
\end{align*}
\item[({\bf L3})] There exists a compact set $\mathcal{K}_{L} \subset \R^{d}$ and a constant $\delta_{L} > 0 $ such that
	\begin{equation}\label{eq:gap}
		 \inf_{x \in \R^{d} \backslash \mathcal{K}_{L}} L(x,0) \geq \delta_{L}  + \min_{x \in \mathcal{K}_{L}} L(x,0).
	\end{equation} 
\end{itemize}

In view of {\bf (L3)} there exists $x^{*} \in \K_{L}$ such that $L(x^{*},0)=\displaystyle{\min_{x \in \K_{L}}} L(x,0)$. Thus, by {\bf (L1)} and {\bf (L2)}  we have that
\begin{equation}\label{eq:L0}
	 \frac{1}{2\ell_{1}}|u|^{2} + L(x^{*},0) \leq L(x,u) \leq C_{1}(1+|x|)(1+|u|^{2}), \quad \forall\ (x,u) \in \R^{d} \times \R^{m}.
\end{equation}
Hereafter we fix $x^{*} \in \K_{L}$ as above and we set 
\begin{equation*}
\delta(x):=d_{\SR}(x,x^{*}). 
\end{equation*}
Let $H: \R^{d} \times \R^{d} \to \R$ be the Hamiltonian associated with $L$, that is, 
\begin{equation}\label{eq:Hamiltonian}
H(x,p)=\sup_{u \in \R^{m}} \left\{ \displaystyle{\sum_{i=1}^{m}} u_{i} \langle p, f_{i}(x)\rangle  - L(x,u) \right\}, \quad \forall\ (x,p) \in \R^{d} \times \R^{d}.
	\end{equation}
	
	We now recall the main results of \cite{bib:CM02} which are the starting point of the current analysis. Consider the following minimization problem: for any $T > 0$ and any $x \in \R^{d}$
 \begin{equation}\label{eq:minimization}
 \text{to minimize}\ \int_{0}^{T}{L(\gamma(s), u(s))\ ds}\,\, \text{over all}\ (\gamma, u) \in \Gamma_{0,T}^{x \to}, 
 \end{equation}
and define the function $V_{T}: \R^{d} \to \R$ by 
\begin{equation}\label{eq:evoValue}
V_{T}(x) = \inf_{(\gamma, u) \in \Gamma_{0,T}^{x \to}} \int_{0}^{T}{L(\gamma(s), u(s))\ ds}, \quad \forall\ x \in \R^{d}.	
\end{equation}
For any $x \in \R^{d}$ we say that a trajectory-control pair $(\gamma, u) \in \Gamma_{0,T}^{x \to}$ is optimal if it solves \eqref{eq:minimization}. Note that the existence of optimal trajectory-control pairs for \eqref{eq:minimization} is a well-known result (see, e.g., \cite[Theorem 7.4.4]{bib:SC}).

By \cite[Theorem 4.2]{bib:CM02} we have that for any $R \geq 0$ there exist two constants $P_{R}$, $Q_{R} \geq 0$ such that, for any $x \in \overline{B}_{R}$, any $T \geq \delta(x)$, and any optimal trajectory-control pair $(\gamma_{x}, u_{x}) \in \Gamma_{0,T}^{x \to}$ for \eqref{eq:minimization} the following holds:
	\begin{equation}\label{eq:l2bounds}
	\int_{0}^{T}{|u_{x}(t)|^{2}\ dt} \leq P_{R}
	\end{equation}
	and
	\begin{equation}\label{eq:compactnesstrajectory}
	|\gamma_{x}(t)| \leq Q_{R}, \quad \forall\ t \in [0,T].	
	\end{equation}	
Moreover, there exists $\mane \in \R$ such that:
\begin{enumerate}
\item for any $R \geq 0$
\begin{equation}\label{eq:maneconstant}
	\lim_{T \to +\infty} \frac{1}{T}V_{T}(x)= \mane, \quad \text{uniformly on}\,\, \overline{B}_{R},
\end{equation}
see \cite[Theorem 5.3]{bib:CM02};
\item the critical equation 
\begin{equation}\label{eq:HJ}
\mane + H(x, D\chi(x))=0, \quad x \in \R^{d},
\end{equation}
has a continuous viscosity solution satisfying \eqref{eq:corrector} below, see \cite[Theorem 5.7]{bib:CM02}. 
\end{enumerate}



Differently from \cite{bib:CM02}, in this paper we also need to assume the following:
\begin{itemize}
	\item[{\bf (S)}] there are no singular minimizing controls of problem \eqref{eq:minimization} different from zero.
	\end{itemize}
The above extra assumption is needed to improve the regularity of viscosity solutions. Indeed, it was proved in \cite[Theorem 1]{bib:CR} that, under {\bf (S)}, the fundamental solution
\begin{equation*}
	A_{t}(x,y)=\inf_{(\gamma, u) \in \Gamma_{0,t}^{x \to y}} \left\{\int_{0}^{t}{L(\gamma(s), u(s))\ ds} \right\}.	
\end{equation*}
 to the critical equation \eqref{eq:HJ} is locally semiconcave as a function of $y \mapsto A_{t}(x,y)$ and, consequently, locally Lipschitz continuous.
 We recall that a function $g: \R^{d} \to \R$ is locally semiconcave if for each compact $\Omega \subset \R^{d}$ there exists a constant $C_{\Omega} \geq 0$ such that 
 \begin{equation*}
 \mu g(x) + (1-\mu) g(y)  - g(\mu x + (1-\mu)y) \leq \mu(1-\mu)C_{\Omega}|x-y|^{2}
 \end{equation*}
for any $\mu \in [0,1]$ and any $x$, $y \in \Omega$. 

\begin{remarks}\label{rem:loclipschitz}\em
\begin{itemize}
	\item[($i$)] Observe that assumption {\bf (S)} is satisfied on sub-Riemannian structures such as the Heisenberg group or Grushin type systems, see for instance \cite[Theorem 5.1]{bib:CR}.
\item[($ii$)] In view of the above assumptions on $L$, we deduce that for any $R \geq 0$ there exists a constant $C_{R} \geq 0$ such that
	\begin{equation}\label{eq:Hcondition}
	|H(x,p) - H(y,p)| \leq C_{R}(1+|p|^{2})|x-y|, \quad \forall\ x, y \in \overline{B}_{R}.	
	\end{equation}
We give the proof of \eqref{eq:Hcondition} for the reader's convenience. First, we claim that, if $u_{x}$ is a maximizer of \eqref{eq:Hamiltonian} associated with $(x,p) \in \R^{d} \times \R^{d}$ then 
	\begin{equation}\label{eq:claim}
	|u_{x}| \leq C(|x|, |p|).	
	\end{equation}
	Indeed, on the one hand, we know that
$
H(x,p) = \langle u_{x}, F^{\star}(x)p \rangle  - L(x,u_{x}).	
$
On the other hand, by definition we have that 
$
H(x,p) \geq - L(x,0).	
$
Therefore, we deduce that
\begin{equation}\label{eq:cri}
L(x,u_{x}) \leq \langle u_{x}, F^{\star}(x)p \rangle + L(x,0) \leq \frac{1}{4\ell_{1}}|u_{x}|^{2} + \ell_{1}^{2}|F^{\star}(x)p|^{2} + C_{1}(1+ |x|). 
\end{equation}
	Since by \eqref{eq:L0} we have that 
	\begin{equation}\label{eq:cri1}
	L(x,u_{x}) \geq \frac{1}{2\ell_{1}}|u_{x}|^{2} + L(x^{*},0),
	\end{equation}
combining \eqref{eq:cri} and \eqref{eq:cri1} we get \eqref{eq:claim}.

We can now prove \eqref{eq:Hcondition}. Fix $R \geq 0$ and for $x \in \overline{B}_{R}$ let $u_{x}$ be a maximizer of \eqref{eq:Hamiltonian}. Then, we have that 
\begin{align*}
H(y,p) - H(x,p) \leq\ & \langle \big(F(y)-F(x)\big)u_{x},p \rangle + L(x,u_{x}) - L(y,u_{x}) 
\\
\leq & C_{1}|y-x||u_{x}||p| + C_{2}|y-x|(|u_{x}|^{2} + 1)	
\end{align*}
which by \eqref{eq:claim} and {\bf (L2)} yields to \eqref{eq:Hcondition}. \qed
\end{itemize}
\end{remarks}

\section{Characterization of the ergodic constant}
\label{sec:ergodicconstant}
 We begin by introducing a class of probability measures that adapts the notion of closed measures to sub-Riemannian control systems.
Set
\begin{equation*}
\PP_{c}^{2}(\R^{d} \times \R^{m}) = \left\{\mu \in \PP(\R^{d} \times \R^{m}) : \int_{\R^{d} \times \R^{m}}{|u|^{2}\ \mu(dx,du)} < +\infty,\ \supp(\pi_{1} \sharp \mu)\, \text{compact} \right\}	
\end{equation*}
where $\pi_{1}: \R^{d} \times \R^{m} \to \R^{d}$ denotes projection onto the first factor, i.e. $\pi_{1}(x,u)= x$. 

Recall that $F(x)=[f_{1}(x)| \dots | f_{m}(x)]$ is the real $d \times m$ matrix introduced in \eqref{eq:dynamics}.

\begin{definition}[{\bf $F$-closed measure}]\label{def:Closedmeasure}
We say that $\mu \in \PP_{c}^{2}(\R^{d} \times \R^{m})$ is an F-closed measure if
	\begin{equation*}
	\int_{\R^{d} \times \R^{m}}{\langle F^{\star}(x)D\varphi(x), u \rangle\ \mu(dx,du)}=0, \quad \forall\ \varphi \in C^{1}(\R^{d}). 	
	\end{equation*}
	We denote by $\C_{F}$ the set of all $F$-closed measures.
\end{definition}

Closed measures were first introduced in \cite{bib:FAS} in order to overcome the lack of regularity of the Lagrangian. Indeed, if $L$ is merely continuous, then there is no Euler-Lagrange flow and, consequently, it is not possible to introduce invariant measures as it is customary, see for instance \cite{bib:FA}. Similarly, in our setting, such a flow does not exist: for this reason, the use of closed measures turns out to be necessary. Moreover, as we will show in the next result, such measures collect the behavior of the minimizing trajectories for the infimum in \eqref{eq:evoValue} as $T \to \infty$. 

We now proceed to construct one closed measure that will be particularly useful to study the Aubry set. Given $x_{0} \in \R^{d}$, for any $T > 0$ let the pair $(\gamma_{x_{0}}, u_{x_{0}}) \in \Gamma_{0,T}^{x_{0} \to}$ be optimal for \eqref{eq:minimization}. Define the probability measure $\mu^{T}_{x_{0}}$ by 
\begin{equation}\label{eq:uniformmeasure}
\int_{\R^{d} \times \R^{m}}{\varphi(x,u)\ \mu^{T}_{x_{0}}(dx,du)}= \frac{1}{T}\int_{0}^{T}{\varphi(\gamma_{x_{0}}(t), u_{x_{0}}(t))\ dt}, \quad \forall\ \varphi \in C_{b}(\R^{d} \times \R^{m}).
\end{equation}
Then, we have the following.  
 
\begin{proposition}\label{prop:nonempty}
Assume {\bf (F1)}, {\bf (F2)} and {\bf (L0)} -- {\bf (L3)}.  Then, $\{ \mu^{T}_{x_{0}}\}_{T>0}$ is tight and there exists a sequence $T_{n} \to \infty$ such that $\mu^{T_{n}}_{x_{0}}$ weakly-$^*$ converges to an $F$-closed measure $\mu^{\infty}_{x_{0}}$.  
\end{proposition}

\proof First, from \eqref{eq:compactnesstrajectory} it follows that $\{ \pi_{1} \sharp \mu^{T}_{x_{0}}\}_{T >0}$ has compact support, uniformly in $T$. Thus, such a family of measures is tight. Let us prove that $\{ \pi_{2} \sharp \mu^{T}\}_{T > 0}$ is also tight.	

On the one hand, by the upper bound in \eqref{eq:L0} we have that
	\begin{equation*}
	\frac{1}{T}v^{T}(x_{0}) \leq \frac{1}{T}\int_{0}^{T}{L(x_{0}, 0)\ ds} \leq  \ell_{1}(1+|x_{0}|).
	\end{equation*}
On the other hand, since $(\gamma_{x_{0}}, u_{x_{0}})$ is a minimizing pair for $V_{T}(x_{0})$, by the lower bound in \eqref{eq:L0} we get
\begin{align*}
& \frac{1}{T}v^{T}(x_{0}) = \frac{1}{T}\int_{0}^{T}{L(\gamma_{x_{0}}(t), u_{x_{0}}(t))\ dt} 
\\
= &  \int_{\R^{d} \times \R^{m}}{L(x,u)\ \mu^{T}_{x_{0}}(dx,du)} \geq \int_{\R^{d} \times \R^{m}}{\left(\frac{1}{2\ell_{1}} |u|^{2} + L(x^{*},0)\right)\ \mu^{T}_{x_{0}}(dx,du)}. 	
\end{align*}
Therefore,
\begin{equation*}
\frac{1}{2\ell_{1}}\int_{\R^{d} \times \R^{m}}{|u|^{2}\ \mu^{T}_{x_{0}}(dx,du)} \leq  \ell_{1}(1+|x_{0}|) - L(x^{*},0). 
\end{equation*}
Consequently, the family of probability measures $\{ \pi_{2} \sharp \mu^{T}_{x_{0}}\}_{T >0}$ has bounded (w.r.t. $T$) second order moment. So,  $\{ \pi_{2} \sharp \mu^{T}_{x_{0}}\}_{T >0}$ is tight.

Since $\{\pi_{1} \sharp \mu^{T}_{x_{0}}\}_{T > 0}$ and $\{\pi_{2} \sharp \mu^{T}_{x_{0}}\}_{T > 0}$ are tight, so is $\{\mu^{T}_{x_{0}}\}_{T >0}$ by \cite[Theorem 5.2.2]{bib:AGS}.  Therefore, by Prokhorov's Theorem there exists  $\{T_{n}\}_{n \in \N}$, with $T_{n} \to \infty$, and $\mu^{\infty}_{x_{0}} \in \PP_{c}^{2}(\R^{d} \times \R^{m})$ such that $\mu^{T_{n}}_{x_{0}} \rightharpoonup^{*} \mu^{\infty}_{x_{0}}$.

We now show that $\mu^{\infty}_{x_{0}}$ is an $F$-closed measure, that is,
\begin{equation*}
	\int_{\R^{d} \times \R^{m}}{\langle F^{\star}(x)D\psi(x), u \rangle\ \mu^{\infty}_{x_{0}}(dx,du)}=0 \quad (\psi \in C^{1}(\R^{d})).
\end{equation*}
By definition we have that 
\begin{align*}
& \int_{\R^{d} \times \R^{m}}{\langle F^{\star}(x)D\psi(x), u\rangle \mu_{x_{0}}^{T_{n}}(dx,du)} = \frac{1}{T_{n}}\int_{0}^{T_{n}}{\langle F^{\star}(\gamma_{x_{0}}(t)D\psi(\gamma_{x_{0}}(t)), u_{x_{0}}(t)\rangle\ dt}
\\
=\ & \frac{1}{T_{n}}\int_{0}^{T_{n}}{\langle D\psi(\gamma_{x_{0}}(t), \dot\gamma_{x_{0}}(t)\rangle\ dt} = \frac{\psi(\gamma_{x_{0}}(T_{n})) - \psi(x_{0})}{T_{n}}. 
\end{align*}
Since, $\gamma_{x_{0}}(T_{n}) \in \overline{B}_{Q_{|x_{0}|}}$ in view of \eqref{eq:compactnesstrajectory}, we obtain
\begin{equation*}
\int_{\R^{d} \times \R^{m}}{\langle F^{\star}(x)D\psi(x), u \rangle\ \mu^{\infty}_{x_{0}}(dx,du)} = \lim_{n\to \infty} \frac{\psi(\gamma_{x_{0}}(T_{n})) - \psi(x_{0})}{T_{n}} = 0. \eqno\square
\end{equation*}

		\medskip

	The following property, which is interesting in its own right, will be crucial for the characterization of the critical constant derived in \Cref{thm:main1} below. For any $R \geq 0$, we set
	\begin{equation*}
	\PP_{R}^{2}(\R^{d} \times \R^{m}) = \left\{\mu \in \PP^{2}_{c}(\R^{d} \times \R^{m}) : \supp(\pi_{1} \sharp \mu)\subset \overline{B}_{R} \right\}.
\end{equation*}

	\begin{proposition}\label{prop:minimax}
	Assume {\bf (F1)}, {\bf (F2)} and {\bf (L0)} -- {\bf (L3)}. Then, for any $R \geq Q_{0}$, where $Q_{0}$ is given in \eqref{eq:compactnesstrajectory}, we have that
	\begin{equation}\label{eq:RR}
	\inf_{\mu \in \C_{F} \cap \PP_{R}^{2}(\R^{d} \times \R^{m})} \int_{\R^{d} \times \R^{m}}{L(x,u)\ \mu(dx,du)}=-\inf_{\psi \in C^{1}(\R^{d})} \sup_{x \in \overline{B}_{R}} H(x, D\psi(x)). 	
	\end{equation}
	\end{proposition}
	
	The following two lemmas are needed for the proof of \Cref{prop:minimax}.

	\begin{lemma}\label{lem:minmaxprel}
		Assume {\bf (F1)}, {\bf (F2)} and {\bf (L0)} -- {\bf (L3)}. Then, for any $R \geq Q_{0}$, where $Q_{0}$ is given in \eqref{eq:compactnesstrajectory}, we have that 
		\begin{align}\label{eq:hand1}
	\begin{split}
	& \inf_{\mu \in \C_{F} \cap \PP_{R}^{2}(\R^{d} \times \R^{m})} \int_{\R^{d} \times \R^{m}}{L(x,u)\ \mu(dx,du)}	
	\\
	= &  \inf_{\mu \in \PP_{R}^{2}(\R^{d} \times \R^{m})}\sup_{\psi \in C^{1}(\R^{d})}\int_{\R^{d} \times \R^{m}}{\Big(L(x,u) - \langle F^{\star}(x)D\psi(x),u \rangle  \Big)\ \mu(dx,du)}.
	\end{split}
	\end{align}
   \end{lemma}

\noindent The proof of the above lemma is based on an argument which is quite common in optimal transport theory (see, e.g., \cite[Theorem 1.3]{bib:CV}). We give the reasoning for the reader's convenience. 

\proof
Since $L$ is bounded below we have that
	\begin{align*}
		& \inf_{\mu \in \C_{F} \cap \PP_{R}^{2}(\R^{d} \times \R^{m})} \int_{\R^{d} \times \R^{m}}{L(x,u)\ \mu(dx,du)}	
	\\
	= &  \inf_{\mu \in \PP_{R}^{2}(\R^{d} \times \R^{m})}\left\{\int_{\R^{d} \times \R^{m}}{L(x,u)\ \mu(dx,du)} + \omega(\mu)\right\}
	\end{align*}
where 
\begin{align*}
	\omega(\mu)=
	\begin{cases}
		0, & \quad \mu \in \C_{F}
		\\
		+\infty, & \quad \mu \not\in \C_{F}.
	\end{cases}
\end{align*}
So, observing that
\begin{equation*}
	\omega(\mu) = \sup_{\psi \in C^{1}(\R^{d})}-\int_{\R^{d} \times \R^{m}}{\langle F^{\star}(x)D\psi(x), u\rangle\ \mu(dx,du)} 
\end{equation*}
we obtain \eqref{eq:hand1}. \qed


\begin{lemma}\label{lem:weakconv}
Let $\phi \in C(\R^{d} \times \R^{m})$ be such that $$\phi_{0} \leq \phi(x,u) \leq C_{\phi}(1+|u|^{2}), \quad \forall\ (x,u) \in \R^{d} \times \R^{m}$$
for some constants $\phi_{0} \in \R$ and $C_{\phi} \geq 0$. Let $\{\mu_{j}\}_{j \in \N} \in \PP^{2}(\R^{d} \times \R^{m})$ and let $\mu \in \PP^{2}(\R^{d} \times \R^{m})$ be such that $\mu_{j} \rightharpoonup^{*} \mu$ as $j \to \infty$. Then, we have that
\begin{equation}\label{eq:liminfineq}
	\liminf_{j \to \infty} \int_{\R^{d} \times \R^{m}}{\phi(x,u)\ \mu_{j}(dx,du)} \geq \int_{\R^{d} \times \R^{m}}{\phi(x,u)\ \mu(dx,du)}. 
\end{equation}
\end{lemma}
\proof
We first prove \eqref{eq:liminfineq} assuming $\phi_{0} = 0$. For any $\eps > 0$ we have that
		\begin{align*}
		& \int_{\R^{d} \times \R^{m}}{\phi(x,u)\ \mu^{T_{j}}_{x_{0}}(dx,du)} = 		\int_{\R^{d} \times \R^{m}}{\frac{\phi(x,u)}{1+\eps|u|^{2}}(1+\eps|u|^{2})\ \mu^{T_{j}}_{x_{0}}(dx,du)}
		\\
		\geq &  \int_{\R^{d} \times \R^{m}}{\frac{\phi(x,u)}{1+\eps|u|^{2}}\ \mu^{T_{j}}_{x_{0}}(dx,du)}.
		\end{align*}
       The growth assumption on $\phi$ ensures that the function $(x, u) \mapsto \frac{\phi(x,u)}{1+\eps|u|^{2}}$ is bounded. So,
       \begin{equation*}
       	\liminf_{j \to +\infty}\int_{\R^{d} \times \R^{m}}{\phi(x,u)\ \mu^{T_{j}}_{x_{0}}(dx,du)} \geq \int_{\R^{d} \times \R^{m}}{\frac{\phi(x,u)}{1+\eps|u|^{2}}\ \mu^{\infty}_{x_{0}}(dx,du)}.
       \end{equation*}
Therefore, as $\eps \downarrow 0$ we obtain \eqref{eq:liminfineq}. 

Now, in order to treat the general case of $\phi_{0} \not= 0$, we observe that
\begin{align*}
		& \int_{\R^{d} \times \R^{m}}{[(\phi(x,u)-\phi_{0}) + \phi_{0}]\ \mu^{T_{j}}_{x_{0}}(dx,du)} 
		\\
		= & 	\int_{\R^{d} \times \R^{m}}{\frac{\phi(x,u)-\phi_{0}}{1+\eps|u|^{2}}(1+\eps|u|^{2})\ \mu^{T_{j}}_{x_{0}}(dx,du)} + \phi_{0}
		\\
		\geq &  \int_{\R^{d} \times \R^{m}}{\frac{\phi(x,u)-\phi_{0}}{1+\eps|u|^{2}}\ \mu^{T_{j}}_{x_{0}}(dx,du)}+ \phi_{0}.
		\end{align*}
Thus, we obtain
 \begin{equation*}
       	\liminf_{j \to +\infty}\int_{\R^{d} \times \R^{m}}{\phi(x,u)\ \mu^{T_{j}}_{x_{0}}(dx,du)} \geq \int_{\R^{d} \times \R^{m}}{\frac{\phi(x,u)-\phi_{0}}{1+\eps|u|^{2}}\ \mu^{\infty}_{x_{0}}(dx,du)} + \phi_{0}
       \end{equation*} 
       which in turn yields the result as $\eps \downarrow 0$.
 \qed

\medskip
\noindent{\it Proof of \Cref{prop:minimax}.}
We divide the proof into two steps. 

\medskip
\noindent {\bf (1).}  We will prove that 
	\begin{align}\label{eq:claimminimax}
	\begin{split}
		&  \inf_{\mu \in \PP_{R}^{2}(\R^{d} \times \R^{m})}\sup_{\psi \in C^{1}(\R^{d})}\int_{\R^{d} \times \R^{m}}{\Big(L(x,u)- \langle F^{\star}(x)D\psi(x),u \rangle  \Big)\ \mu(dx,du)}
		\\
		= &  \sup_{\psi \in C^{1}(\R^{d})} \inf_{\mu \in \PP_{R}^{2}(\R^{d} \times \R^{m})} \int_{\R^{d} \times \R^{m}}{\Big(L(x,u)- \langle F^{\star}(x)D\psi(x),u \rangle  \Big)\ \mu(dx,du)}. 
		\end{split}
	\end{align}
For the proof of \eqref{eq:claimminimax} we will apply the Minimax Theorem (\cite[Theorem A.1]{bib:OPA}) to the functional 
\[
\mathcal{F}: C^{1}(\R^{d}) \times \PP_{R}^{2}(\R^{d} \times \R^{m}) \to \R
\]
defined by
	\begin{equation*}
	\mathcal{F}(\psi, \mu) = \int_{\R^{d} \times \R^{m}}{\Big(L(x,u)-\langle F^{\star}(x)D\psi(x), u \rangle \Big)\ \mu(dx,du)}.
	\end{equation*}
	To do so, define
	\begin{align*}
	c^{*}=1+ L(x^{*},0). 	
	\end{align*}
In order to check that the assumptions of the Minimax Theorem are satisfied we have to prove the following.
 \begin{itemize}
 \item[($i$)] $\mathcal{E}:=\big\{\mu \in \PP_{R}^{2}(\R^{d} \times \R^{m}):\ \mathcal{F}(0,\mu) \leq c^{*} \big\}$ is compact in  $\big(\PP^{2}_{R}(\R^{d} \times \R^{m}), d_{1}\big)$.   Indeed, for any given $\mu \in \PP_{R}^{2}(\R^{d} \times \R^{m})$ we know that the support of $\pi_{1} \sharp \mu$ is contained in $\overline{B}_{R}$. Moreover, the coercivity of $L$ implies that for any $\mu \in \mathcal{E}$ we have that the family $\{\pi_{2} \sharp \mu\}_{\mu \in \mathcal{E}}$ has bounded second moment and, therefore, it is tight. Thus, $\mu$ is tight by \cite[Theorem 5.2.2]{bib:AGS} and  $\mathcal{E}$ is compact by Prokhorov's Theorem.
 \item[($ii$)] The map $\mu \mapsto \mathcal{F}(\psi, \mu)$ is lower semicontinuous on $\mathcal{E}$ for any $\psi \in C^{1}(\R^{d})$. This fact follows from \Cref{lem:weakconv}.
 \end{itemize}

Therefore, by the Minimax Theorem we  obtain \eqref{eq:claimminimax}.	
	
	\medskip
\noindent{\bf (2).} In order to complete the proof we observe that, in view of \eqref{eq:hand1} and \eqref{eq:claimminimax}, 
 \begin{align*}
 	& \inf_{\mu \in \C_{F} \cap \PP_{R}^{2}(\R^{d} \times \R^{m})} \int_{\R^{d} \times \R^{m}}{L(x,u)\ \mu(dx,du)}
 	\\
 	= & \inf_{\mu \in \PP_{R}^{2}(\R^{d} \times \R^{m})}\sup_{\psi \in C^{1}(\R^{d})}\int_{\R^{d} \times \R^{m}}{\Big(L(x,u)- \langle F^{\star}(x)D\psi(x),u \rangle  \Big)\ \mu(dx,du)}
 	\\
 	= &  \sup_{\psi \in C^{1}(\R^{d})} \inf_{\mu \in \PP_{R}^{2}(\R^{d} \times \R^{m})} \int_{\R^{d} \times \R^{m}}{\Big(L(x,u)- \langle F^{\star}(x)D\psi(x),u \rangle  \Big)\ \mu(dx,du)}.
\end{align*}
Now, the coercivity of $L$ ensures the existence of 
\begin{equation*}
 \min_{(x,u) \in \overline{B}_{R} \times \R^{m}}\Big\{ L(x,u)- \langle F^{\star}(x)D\psi(x),u \rangle \Big\}.
\end{equation*}
Therefore, by taking a Dirac mass centered at any pair $(\overline{x}, \overline{u})$ at which the above minimum is attained, one deduces that
\begin{align*}
&  \sup_{\psi \in C^{1}(\R^{d})} \inf_{\mu \in \PP_{R}^{2}(\R^{d} \times \R^{m})} \int_{\R^{d} \times \R^{m}}{\Big(L(x,u)- \langle F^{\star}(x)D\psi(x),u \rangle  \Big)\ \mu(dx,du)}
\\
 =\	& \sup_{\psi \in C^{1}(\R^{d})} \min_{(x,u) \in \overline{B}_{R} \times \R^{m}}\Big\{ L(x,u)- \langle F^{\star}(x)D\psi(x),u \rangle \Big\}.	
\\
	 	=\ & \sup_{\psi \in C^{1}(\R^{d})} \left(-\max_{(x,u) \in \overline{B}_{R} \times \R^{m}}\Big\{ \langle F^{\star}(x)D\psi(x),u \rangle - L(x,u)  \Big\}\right)
	\\
	=\ & - \inf_{\psi \in C^{1}(\R^{d})} \max_{(x,u) \in \overline{B}_{R} \times \R^{m}}\Big\{  \langle F^{\star}(x)D\psi(x),u \rangle - L(x,u) \Big\}
	\\
 	= &   - \inf_{\psi \in C^{1}(\R^{d})} \max_{ x \in \overline{B}_{R}}\  H(x, D\psi(x)).
 \end{align*}
We observe that the last inequality holds because 
\begin{align*}
& \max_{(x,u) \in \overline{B}_{R} \times \R^{m}}\Big\{\langle F^{\star}(x)D\psi(x), u \rangle - L(x,u) \Big\}
\\
=\ & \max_{x \in \overline{B}_{R}} \sup_{u \in \R^{m}}\Big\{\langle F^{\star}(x)D\psi(x), u \rangle - L(x,u) \Big\} =  \max_{x \in \overline{B}_{R}} H(x, D\psi(x)). 
\end{align*}
This completes the proof. \qed


\medskip

The following characterization of the critical value is essential for the analysis in \Cref{sec:aubrymather}.

\begin{theorem}\label{thm:main1}
Assume {\bf (F1)}, {\bf (F2)}, {\bf (L0)} -- {\bf (L3)} and {\bf (S)}. Then, we have that
\begin{equation}\label{eq:mumane}
\mane  = \min_{\mu \in \C_{F}} \int_{\R^{d} \times \R^{m}}{L(x,u)\ \mu(dx,du)}. 
\end{equation}
\end{theorem}

Once again, we need a lemma for the proof of \Cref{thm:main1}. 

\begin{lemma}\label{lem:smoothapprox}
Assume {\bf (F1)}, {\bf (F2)}, {\bf (L0)} -- {\bf (L3)} and {\bf (S)}. Let $\chi$ be a critical solution. Then, for any $R \geq 0$ there exists a constant $\kappa_{R} \geq 0$ such that, for any $\eps > 0$
\begin{equation}\label{eq:smoothsub}
	\mane + H(x, D\chi_{\eps}(x)) \leq \kappa_{R}\eps \quad \forall\ x \in \overline{B}_{R},
\end{equation}
where $\chi_{\eps}(x)= \chi \star \xi_{\eps}(x)$ and $\xi_{\eps}$ is a smooth mollifier. 
\end{lemma}
\proof
As observed in \Cref{rem:loclipschitz} by {\bf (S)} we have that $\chi$ belongs to $W^{1,\infty}_{\text{loc}}(\R^{d})$. So,
\begin{equation}\label{eq:almostev}
\mane + H(x, D\chi(x)) = 0, \quad \text{a.e.}\ x \in \R^{d}.	
\end{equation}

Let $R \geq 0$ and let $x_{0} \in \overline{B}_{R}$. Then, by Jensen's inequality  we get
\begin{align*}
\mane + H(x_{0}, D\chi_{\eps}(x_{0})) =&  \mane + H\left(x_{0}, \int_{\R^{d}}{D\chi(x_{0}-y) \xi_{\eps}(y)\ dy}\right) 
\\
\leq & \int_{\R^{d}}{\big[\mane + H(x_{0}, D\chi(x_{0}-y)) \big]\xi_{\eps}(y)\ dy}.  	
\end{align*}
Moreover, 
\begin{align*}
	& \int_{\R^{d}}{\big[\mane + H(x_{0}, D\chi(x_{0}-y)) \big]\xi_{\eps}(y)\ dy} 
	\\
	= & \underbrace{\int_{\R^{d}}{\big[\mane + H(x_{0}-y, D\chi(x_{0}-y)) \big]\xi_{\eps}(y)\ dy}}_{\bf I} 
	\\
	+ & \underbrace{\int_{\R^{d}}{\big[H(x_{0}, D\chi(x_{0}-y)) - H(x_{0}-y, D\chi(x_{0}-y))  \big]\xi_{\eps}(y)\ dy}}_{\bf II}.
\end{align*}
Since, ${\bf I}=0$  by \eqref{eq:almostev} and ${\bf II} \leq \kappa_{R} \eps$ for some constant $\kappa_{R} \geq 0$ on account of \eqref{eq:Hcondition}, we deduce \eqref{eq:smoothsub} because $x_{0}$ is an arbitrary point in $\overline{B}_{R}$. \qed

\medskip
\noindent {\it Proof of \Cref{thm:main1}.}
	We divide the proof into two steps. 
	
	{\bf Step 1.}
We first show that for any $R \geq Q_{0}$, where $Q_{0}$ is given in \eqref{eq:compactnesstrajectory},
\begin{equation*}
\mane = \min_{\mu \in \C_{F} \cap \PP_{R}^{2}(\R^{d} \times \R^{m})}	\int_{\R^{d} \times \R^{m}}{L(x,u)\ \mu(dx,du)}.
\end{equation*}
We observe that the above minimum exists since $\C_{F} \cap \PP_{R}^{2}(\R^{d} \times \R^{m})$ is compact w.r.t. $d_{1}$ distance. 
By \eqref{eq:maneconstant} we have that 
 \begin{align*}
 \mane = & \lim_{T \to +\infty} \frac{1}{T} v^{T}(0).  
 \end{align*}
 Hence, appealing to \Cref{lem:weakconv} and recalling that $L(x,u) \geq L(x^{*},0)$ we obtain
 \begin{align}\label{eq:star}
 \begin{split}
  \mane = & \lim_{T \to \infty}\int_{\R^{d} \times \R^{m}}{L(x,u)\ \mu^{T}_{0}(dx,du)} \geq  \int_{\R^{d} \times \R^{m}}{L(x,u)\ \mu^{\infty}_{0}(dx,du)}.
 \end{split} 
 \end{align}
Here, we recall that $\mu^{T}_{0}$ is defined by
\begin{equation*}
\int_{\R^{d} \times \R^{m}}{\varphi(x,u)\ \mu^{T}_{0}(dx,du)}= \frac{1}{T}\int_{0}^{T}{\varphi(\gamma_{0}(t), u_{0}(t))\ dt}, \quad \forall\ \varphi \in C_{b}(\R^{d} \times \R^{m})
\end{equation*}
with $(\gamma_{0}, u_{0}) \in \Gamma_{0,T}^{0 \to}$ optimal for \eqref{eq:minimization} and $\mu^{\infty}_{0}$ denotes the limit measure of $\{\mu^{T}_{0}\}_{T >0}$ constructed in \Cref{prop:nonempty}. So, since $\mu_{0}^{\infty} \in \C_{F} \cap \PP_{R}^{2}(\R^{d} \times \R^{m})$ for any $R \geq Q_{0}$, we obtain
\begin{equation*}
\mane \geq \min_{\mu \in \C_{F} \cap \PP_{R}^{2}(\R^{d} \times \R^{m})}\int_{\R^{d} \times \R^{m}}{L(x,u)\ \mu(dx,du)}.	
\end{equation*}
Next, by \Cref{prop:minimax} we have that for any $\psi \in C^{1}_{c}(\R^{d})$
\begin{align}\label{eq:ineqq}
\begin{split}
\min_{\mu \in \C_{F} \cap \PP_{R}^{2}(\R^{d} \times \R^{m})} \int_{\R^{d} \times \R^{m}}{L(x,u)\ \mu(dx,du)} = -\inf_{\psi \in C^{1}(\R^{d})} \sup_{x \in \overline{B}_{R}} H(x, D\psi(x)). 
\end{split}
\end{align}
Let $\chi$ be a critical solution. For $\eps \geq 0$ let $\chi_{\eps}(x)= \chi \star \xi^{\eps}(x)$, where $\xi^{\eps}$ is a smooth mollifier. Now, for any $R \geq 0$ \Cref{lem:smoothapprox} yields 
\begin{equation*}
\mane + H(x, D\chi_{\eps}(x)) \leq \kappa_{R}\eps, \quad x \in \overline{B}_{R}. 	
\end{equation*}
Then, using $\chi_{\eps}$ to give a lower bound for the right hand side of \eqref{eq:ineqq} we obtain
\begin{align*}
\min_{\mu \in \C_{F}\cap \PP_{R}^{2}(\R^{d} \times \R^{m})} \int_{\R^{d} \times \R^{m}}{L(x,u)\ \mu(dx,du)}
\geq\  - \sup_{x \in \overline{B}_{R}} H(x, D\chi_{\eps}(x)) \geq \mane - \kappa_{R}\eps.
\end{align*}
Hence, as $\eps \downarrow 0$ we get 
\begin{equation*}
\min_{\mu \in \C_{F}\cap \PP_{R}^{2}(\R^{d} \times \R^{m})} \int_{\R^{d} \times \R^{m}}{L(x,u)\ \mu(dx,du)} \geq \mane 	
\end{equation*}
and this completes the proof of step 1.

{\bf Step 2.} 
In order to prove \eqref{eq:mumane} we must remove the restriction $\mu \in \C_{F}\cap \PP_{R}^{2}(\R^{d} \times \R^{m})$ assumed to obtain the conclusion of step 1.
Let $\{\mu_{j}\}_{j \in \N} \subset \C_{F}$ be such that
\begin{equation}\label{eq:help}
\lim_{j \to \infty} \int_{\R^{d} \times \R^{m}}{L(x,u)\ \mu_{j}(dx,du)} = \inf_{\mu \in \C_{F}} \int_{\R^{d} \times \R^{m}}{L(x,u)\ \mu(dx,du)}.	
\end{equation}
Since $\mu_{j} \in \C_{F} \subset \PP_{c}^{2}(\R^{d} \times \R^{m})$ we have that $\supp(\mu_{j}) \subset \overline{B}_{R_{j}}$ for some $\{R_{j}\}_{j \in \N}$.
Thus, without loss of generality, we can assume that for any $j \in \N$
\begin{equation*}
\min_{\mu \in \C_{F} \cap \PP_{R_{j}}^{2}(\R^{d} \times \R^{m})} \int_{\R^{d} \times \R^{m}}{L(x,u)\ \mu(dx,du)} = \int_{\R^{d} \times \R^{m}}{L(x,u)\ \mu_{j}(dx,du)}. 	
\end{equation*}
So, for $j$ sufficiently large, in view of Step 1 we get
\begin{equation*}
\mane = \int_{\R^{d} \times \R^{m}}{L(x,u)\ \mu_{j}(dx,du)}	
\end{equation*}
and the conclusion follows from \eqref{eq:help}. \qed

\begin{corollary}\label{cor:carmin}
Assume {\bf (F1)}, {\bf (F2)}, {\bf (L0)} -- {\bf (L3)} and {\bf (S)}. Then the following holds.
\begin{itemize}
\item[($i$)] $\mane = L(x^{*}, 0) =  \displaystyle{\min_{x \in \K_{L}}} L(x,0)$. 
\item[($ii$)] For any $x_{0} \in \R^{d}$ we have that
\begin{equation*}
\mane = \int_{\R^{d} \times \R^{m}}{L(x,u)\ \mu^{\infty}_{x_{0}}(dx,du)}	
\end{equation*}
where $\mu^{\infty}_{x_{0}}$ is given in \Cref{prop:nonempty}.
\end{itemize}
\end{corollary}
\begin{remarks}\em
Note that point ($i$) of the conclusion has beed already proved in \cite[Corollary 5.4]{bib:CM02}. Here we propose a different approach which relies on \eqref{eq:mumane}. 
\end{remarks}

\proof
	{\it ($i$)} On the one hand, by \Cref{thm:main1},  we have that
	\begin{equation*}
	\mane = \min_{\mu \in \C_{F}} \int_{\R^{d} \times \R^{m}}{L(x,u)\ \mu(dx,du)} \geq L(x^{*},0)
	\end{equation*}
where the inequality holds true by \eqref{eq:L0}. 

On the other hand, observing that the Dirac measure $\delta_{(x^{*},0)}$ is $F$-closed, we obtain
\begin{equation*}
\mane= \min_{\mu \in \C_{F}} \int_{\R^{d} \times \R^{m}}{L(x,u)\ \mu(dx,du)} \leq \int_{\R^{d} \times \R^{m}}{L(x,u)\ \delta_{(x^{*}, 0)}(dx,du)} = L(x^{*}, 0).	
\end{equation*}

\noindent{\it ($ii$)} Recalling \Cref{lem:weakconv} we obtain
 \begin{align*}
 \begin{split}
  \mane = & \lim_{T \to \infty}\int_{\R^{d} \times \R^{m}}{L(x,u)\ \mu^{T}_{0}(dx,du)} \geq \int_{\R^{d} \times \R^{m}}{L(x,u)\ \mu^{\infty}_{0}(dx,du)}.
 \end{split} 
 \end{align*}
 Thus, the conclusion follows from \Cref{thm:main1} since $\mu^{\infty}_{0}$ is  $F$-closed by \Cref{prop:nonempty}.  \qed

\section{Aubry set}
\label{sec:aubrymather}

\noindent{\it Throughout this section we assume that {\bf (F1)}, {\bf (F2)}, {\bf (L1)} -- {\bf (L3)}, and {\bf (S)} are in force.}

	We denote by $L^{*}$ the Legendre transform of $L$, that is,
	\begin{equation*}
	L^{*}(x,p)= \sup_{u \in \R^{m}} \big\{\langle p,v \rangle - L(x,u) \big\},
	\end{equation*}
and we observe that
\begin{equation}\label{eq:starhj}
 H(x, p) = L^{*}(x, F^{*}(x)p), \quad (x, p) \in \R^{d} \times \R^{d}.
\end{equation}
Since $L$ satisfies {\bf (L0)} -- {\bf (L2)}, we have that $L^{*}$ is coercive and strictly convex in $p$.

		\begin{definition}[{\bf Dominated functions and calibrated curves}]\label{def:domcal}
			Let $c \in \R$ and let $\varphi$ be a continuous function on $\R^{d}$.			
			\begin{enumerate}
			\item We say that $\varphi$ is dominated by $L-c$, and we denote this by $\varphi \prec L -c$, if for any $a$, $b \in \R$, with $a < b$, and any trajectory-control pair $(\gamma,u) \in \Gamma_{a,b}^{\gamma(a) \to \gamma(b)}$ we have that
			\begin{equation*}
			\varphi(\gamma(b)) - \varphi(\gamma(a)) \leq \int_{a}^{b}{L(\gamma(s), u(s))\ ds} - c\ (b-a).
			\end{equation*}
			\item Fix $a$, $b \in \R$ with $a < b$ and let $(\gamma, u) \in \Gamma_{a,b}^{\gamma(a) \to \gamma(b)}$. We say that $\gamma$ is calibrated for $\varphi$ if
			\begin{equation*}
			\varphi(\gamma(b)) - \varphi(\gamma(a)) = \int_{a}^{b}{L(\gamma(s), u(s))\ ds} - c\ (b-a).
			\end{equation*}
			We denote by $\Cal(\varphi)$ the set of all calibrated curves for $\varphi$.
			\end{enumerate}
		\end{definition}




		For any $t \geq 0$ and for any $x$, $y \in \R^{d}$ we recall that $A_{t}(x,y)$ stands for the action functional, also called fundamental solution of the critical equation, i.e.,
		\begin{equation}\label{eq:actionfunctional}
		A_{t}(x,y)=\inf_{(\gamma, u) \in \Gamma_{0,t}^{x \to y}} \int_{0}^{t}{L(\gamma(s), u(s))\ ds} .	
		\end{equation}
 We note that $\varphi \prec L-\mane$ if and only if for any $x$, $y$ in $\R^{d}$ and for any $t \geq 0$ we have that 
\begin{equation}\label{eq:subsol}
	\varphi(y)-\varphi(x) \leq A_{t}(x,y) - \mane t.
\end{equation}
Peierls barrier is defined as
\begin{align}\label{eq:PB}
\PB(x,y) =  \liminf_{t \to \infty}\ \big[A_{t}(x,y) - \mane t\big] \quad (x,y \in \R^{d}).	
\end{align}

\begin{lemma}\label{lemma}
The following properties hold.
\begin{itemize}
\item[($i$)] For any $x$, $y \in \R^{d}$ we have that $0 \leq \PB(x,y) < \infty$.
\item[($ii$)] For any $x$, $y$, $z \in \R^{d}$ there holds
\begin{equation}\label{eq:peierls}
\PB(x,z) \leq \PB(x,y) + \PB(y,z)
\end{equation}
and, 
\begin{equation}\label{eq:trian}
\PB(x,z) \leq \PB(x,y) + A_{t}(y,z) - \mane t \quad (t \geq 0).
\end{equation}
\end{itemize}
\end{lemma}
\proof
In order to prove the lower bound in ($i$), it suffices to observe that by \eqref{eq:L0} and \Cref{cor:carmin} we have that
\begin{equation*}
A_{t}(x,y)-\mane t \geq \int_{0}^{t}{L(x^{*},0)\ ds} - \mane t = 0.
\end{equation*}
Next, in order to show that $\PB(x,y) < \infty$ given $x$, $y \in \R^{d}$, let $(\gamma_{x}, u_{x}) \in \Gamma_{0,\delta^{*}(x)}^{x \to x^{*}}$, $(\gamma_{y}, u_{y}) \in \Gamma_{0,\delta^{*}(y)}^{x^{*} \to y}$ be geodesic pairs. 
Fix $t > \delta^{*}(x) + \delta^{*}(y)$ and define the control
\begin{equation*}
\widetilde{u}(s)=
\begin{cases}
u_{x}(s), & s \in [0,\delta^{*}(x)]
\\
0,  & s \in (\delta^{*}(x), t-\delta^{*}(y)]
\\
u_{y}(s-t+\delta^{*}(y)),  & s \in (t-\delta^{*}(y), t].
\end{cases}
\end{equation*}
Let  $\widetilde\gamma$ be such that $(\widetilde\gamma, \widetilde{u}) \in \Gamma_{0,t}^{x \to y}$. Then, we have that 
\begin{multline*}
A_{t}(x,y) - \mane t \leq \int_{0}^{t} L(\widetilde\gamma(s), \widetilde{u}(s))\ ds - \mane t
\\
\leq \int_{0}^{\delta^{*}(x)} L(\gamma_{x}(s), u_{x}(s))\ ds + \int_{\delta^{*}(x)}^{t-\delta^{*}(y)} L(x^{*}, 0)\ ds + \int_{0}^{\delta^{*}(y)} L(\gamma_{y}(s), u_{y}(s))\ ds - \mane t 
\end{multline*}
By definition of geodesic pairs we have that $|u_{x}(s)| \leq 1$ for any $s \in [0,\delta^{*}(x)]$ and $|u_{y}(s)| \leq 1$ for any $s \in [0,\delta^{*}(y)]$. Hence, from \eqref{lem:boundedtrajectories} we obtain
\begin{equation*}
|\gamma_{x}(s)| \leq (|x| + c_{f}\delta^{*}(x))e^{c_{f}\delta^{*}(x)}, \quad s \in [0,\delta^{*}(x)]
\end{equation*}
and
\begin{equation*}
|\gamma_{y}(s)| \leq (|y| + c_{f}\delta^{*}(y))e^{c_{f}\delta^{*}(y)}, \quad s \in [0,\delta^{*}(y)].
\end{equation*}
So, by \eqref{eq:L0} we obtain 
\begin{multline*}
\int_{0}^{\delta^{*}(x)} L(\gamma_{x}(s), u_{x}(s))\ ds + \int_{0}^{\delta^{*}(y)} L(\gamma_{y}(s), u_{y}(s))\ ds 
\\ \leq (|x| + c_{f}\delta^{*}(x))e^{c_{f}\delta^{*}(x)} + (|y| + c_{f}\delta^{*}(y))e^{c_{f}\delta^{*}(y)}.
\end{multline*}
Moreover, by \Cref{cor:carmin} we have that 
\begin{equation*}
\int_{\delta^{*}(x)}^{t-\delta^{*}(y)} L(x^{*}, 0)\ ds - \mane t = - \mane (\delta^{*}(x) + \delta^{*}(y)).
\end{equation*}
Therefore, 
\begin{multline*}
A_{t}(x,y) - \mane t \leq (|x| + c_{f}\delta^{*}(x))e^{c_{f}\delta^{*}(x)} 
\\
+ (|y| + c_{f}\delta^{*}(y))e^{c_{f}\delta^{*}(y)} + \mane (\delta^{*}(x) + \delta^{*}(y)).
\end{multline*}
This yields $\PB(x,y) < \infty$, thus completing the proof of ($i$).

As for ($ii$), we note that the proof of \eqref{eq:peierls} can be found in \cite[Proposition 9.1.2]{bib:FA}. So, let us show \eqref{eq:trian}. Let $x$, $y \in \R^{d}$, let $\eps > 0$ and let $\{s_{n}\}_{n \in \N}$ be such that 
\begin{equation*}
A_{s_{n}}(x,y) -\mane s_{n} \leq \PB(x,y) + \eps, \quad \forall\ n \in \N.
\end{equation*}
Let $t > 0$, $z \in \R^{d}$ and set $t_{n}=s_{n} + t$. Then, we have that 
\begin{align*}
A_{t_{n}}(x,z) -\mane t_{n} \leq\ & A_{s_{n}}(x,y) - \mane s_{n} + A_{t}(y,z) - \mane t
\\
\leq\ & \PB(x,y) + \eps + A_{t}(y,z) - \mane t.
\end{align*}
Hence, we obtain
\begin{equation*}
\PB(x,z) \leq \liminf_{n \to \infty} \{A_{t_{n}}(x,z) - \mane t_{n} \} \leq \PB(x,y) + \eps + A_{t}(y,z) - \mane t
\end{equation*}
which in turn yields \eqref{eq:trian} since $\eps$ is arbitrary. 
\qed

\begin{remarks}\label{rem:peierlsubsolution}\em
Note that, combining \eqref{eq:trian} and \eqref{eq:subsol} we have that Peierls barrier is a dominated by $L-\mane$. Moreover, it is well-known that this implies that $\PB(x,\cdot)$ is a subsolution to the critical equation 
\begin{equation*}
\mane + H(x, D \chi(x))=0 \quad (x \in \R^{d}). 
\end{equation*}
\end{remarks}


\subsection{Properties of Peierls barrier}

In this section we derive several properties of  Peierls barrier. For example, we prove that $\PB(x,\cdot)$ is a critical solution (see \Cref{def:Corrector} below) and we give a construction of calibrated curves for such a function. 

Set
	\begin{equation*}
	\mathcal{S}=\{\varphi \in C(\R^{d}): \varphi \geq 0,\, \varphi \prec L - \mane\}.
	\end{equation*}
	
	\begin{definition}\label{def:Corrector}
	A viscosity solution $\chi \in \mathcal{S}$ of the equation 
	\begin{equation}\label{eq:HJcorrector}
	\mane + H(x, D\chi(x))=0, \quad x \in \R^{d}
	\end{equation}
	is called critical solution if
	\begin{equation}\label{eq:corrector}
	\chi(x)=\inf_{(\gamma, u) \in \Gamma_{0, t}^{\to x}} \left\{\chi(\gamma(0)) + \int_{0}^{t}{L(\gamma(s), u(s))\ ds} \right\} - \mane t, \quad \forall\ t > 0,\, \forall\ x \in \R^{d}.
	\end{equation}
	\end{definition}
	
	Note that \eqref{eq:corrector} can be equivalently written as
	\begin{equation}\label{eq:corrector2}
	\chi(x)=\min_{y \in \R^{d}} \{\chi(y) + A_{t}(y,x)\} - \mane t, \quad \forall t > 0, \, \forall x \in \R^{d}.
	\end{equation}
	
\begin{remarks}\label{rem:Infimum}\em
The existence of critical solutions in $\mathcal{S}$ follows from \cite[Theorem 7.6]{bib:CM02}.  Moreover, by \cite[Theorem 7.3]{bib:CM02}, the infimum in \eqref{eq:corrector} is attained for any $(t, x) \in [0,\infty) \times \R^{d}$. 
\end{remarks}

\begin{lemma}\label{lem:semiconcavity}
Let $\chi$ be a critical solution for \eqref{eq:HJcorrector}. Then, $\chi$ is locally semiconcave. 
\end{lemma}
\proof

Let $y_{t, x} \in \R^{d}$ be the minimum point associated with $(t, x)$ in \eqref{eq:corrector2}. 
Then, for any $t > 0$, $x$, $h \in \R^{d}$ we have that 
\begin{multline*}
\chi(x+h) + \chi(x-h)- 2\chi(x) 
\\
\leq \chi(y_{t, x}) + \chi(y_{t, x}) - 2\chi(y_{t, x}) + A_{t}(y_{t, x}, x+h) + A_{t}(y_{t, x}, x-h) - 2A_{t}(y_{t, x}, x)
\\
= A_{t}(y_{t, x}, x+h) + A_{t}(y_{t, x}, x-h) - 2A_{t}(y_{t, x}, x).
\end{multline*}
The conclusion follows by the semiconcavity of the action function $A_{t}(y_{t, x}, \cdot)$ proved in \cite[Theorem 1]{bib:CR}.
\qed

The following technical lemma is useful for the analysis of the Aubry set.

\begin{lemma}\label{lem:barrierconvergence}
Let $(x,y) \in \R^{2d}$. Let $\{ t_{n}\}_{n \in \N} \in \R$ and $(\gamma_{n}, u_{n}) \in \Gamma_{0,t_{n}}^{x \to y}$ be such that
\begin{equation}\label{eq:zeros}
t_{n} \to +\infty \quad \text{and} \quad \lim_{n \to +\infty} \int_{0}^{t_{n}}{L(\gamma_{u_{n}}(s), u_{n}(s))\ ds}-\mane t_{n} = \PB(x,y). 	
\end{equation}
Then, there exist a subsequence, still denoted by $(\gamma_{n}, u_{n})$, and a pair $(\bar\gamma, \bar{u}) \in \Gamma_{0,\infty}^{x \to}$ such that
\begin{itemize}
\item[($i$)]  $\{u_{n}\}_{n \in \N}$ weakly converges to $\bar{u}$ in $L^{2}$ on any compact subset of $[0,\infty)$;
\item[($ii$)] $\{\gamma_{n}\}_{n \in \N}$ uniformly converges to $\bar\gamma$ on every compact subset of $[0,\infty)$.
\end{itemize} 
\end{lemma}
\proof
For simplicity of notation we set $h:=\PB(x,y)$. From \eqref{eq:PB} it follows that there exists $\bar{n} \in \N$ such that for any $n \geq \overline{n}$ we have that
	\begin{equation*}
		\int_{0}^{t_{n}}{L(\gamma_{n}(s), u_{n}(s))\ ds}-\mane t_{n} \leq h+1.
	\end{equation*}
 On the other hand, by {\bf (L2)} we obtain 
\begin{align*}
\int_{0}^{t_{n}}{L(\gamma_{n}(s), u_{n}(s))\ ds}-\mane t_{n} \geq \frac{1}{2\ell_{1}} \int_{0}^{t_{n}}{|u_{n}(s)|^{2}\ ds} 
\end{align*}
in view of ($i$) in  \Cref{cor:carmin}. So,
\begin{equation*}
\int_{0}^{t_{n}}{|u_{n}(s)|^{2}\ ds} \leq 2\ell_{1}(h+1), \quad \forall n \geq \overline{n}.	
\end{equation*}
Therefore, there exists a subsequence, still denoted by $\{u_{n}\}$, that weakly converges to an admissible control $\bar{u}$ in $L^{2}$ on any compact subset of $[0,+\infty)$. Moreover,
by \eqref{eq:compactnesstrajectory} for any $t > 0 $ we have that
\begin{equation*}
|\gamma_{n}(s)|^{2} \leq Q_{|x|}, \quad \forall\ s \in [0,t],\,\, \forall\ n \geq \overline{n}
\end{equation*}
and, for any $s \in [0,t]$ and any $n \geq \overline{n}$
\begin{align*}
\int_{0}^{t_{n}}|\dot\gamma_{n}(s)|^{2}\ ds \leq \int_{0}^{t_{n}}c_{f}^{2}\big(1 + |\gamma_{n}(s)| \big)^{2}|u_{n}(s)|^{2}\ ds \leq 4c_{f}^{2}(1+Q_{|x|})\ell_{1}(h+1).	
\end{align*}
Hence, $\{\gamma_{n}\}_{n \in \N}$ is uniformly bounded in $W^{1,2}(0,t; \R^{d})$ for any $t > 0$. Then, by the Ascoli-Arzela Theorem, up to extracting a further subsequence,  $\{\gamma_{n}\}_{n \in \N}$ uniformly converges to a curve $\bar{\gamma}$ on every compact subset of $[0,+\infty)$. 

Now, we claim that $(\bar\gamma, \bar{u})$ satisfies \eqref{eq:dynamics}. Indeed, for any $t \geq 0$ we have that 
\begin{equation*}
\gamma_{n}(t)= x + \sum_{i = 1}^{m}{\int_{0}^{t}{u_{i}^{n}(s)f_{i}(\gamma_{n}(s))\ ds}}. 
\end{equation*}
Thus, by the locally uniform convergence of $\gamma_{n}$ it follows that  $f_{i}(\gamma_{n}(t)) \to f_{i}(\bar\gamma(t))$, locally uniformly, for any $t \geq 0$, as $n \to +\infty$ for any $i = 1, \dots, m$. Therefore, taking $v \in \R^{d}$ we deduce that
\begin{equation*}
\langle v, \gamma_{n}(t) \rangle = \langle v, x \rangle + \sum_{i = 1}^{m}{\int_{0}^{t}{ u_{i}^{n}(s)\langle f_{i}(\gamma_{n}(s)), v \rangle \ ds}}, \quad \forall\ t \geq 0.	
\end{equation*}
As $n \to +\infty$ we get
\begin{equation*}
\langle v, \bar\gamma(t) \rangle = \langle v, x \rangle + \sum_{i = 1}^{m}{\int_{0}^{t}{ \bar{u}_{i}(s) \langle f_{i}(\bar\gamma(s)), v \rangle \ ds}}, \quad \forall\ t \geq 0.
\end{equation*}
Since  $v \in \R^{d}$ is arbitrary, the conclusion follows.
\qed

\begin{remarks}\label{rem:conv}\em
Arguing as in the proof of \Cref{lem:barrierconvergence}, one can prove the following. Given $h \in \R$, $\{ t_{n}\}_{n \in \N}$ and $(\gamma_{n}, u_{n}) \in \Gamma_{-t_{n}, 0}^{x \to y}$ such that
\begin{equation*}
t_{n} \to +\infty \quad \text{and} \quad \lim_{n \to +\infty} \int_{-t_{n}}^{0}{L(\gamma_{u_{n}}(s), u_{n}(s))\ ds}-\mane t_{n} = h, 
\end{equation*}
there exists a subsequence, still denoted by $(\gamma_{n}, u_{n})$, and a trajectory-control pair $(\bar\gamma, \bar{u})$ such that
\begin{itemize}
\item[($i$)]  $\{u_{n}\}_{n \in \N}$ weakly converges to $\bar{u}$ in $L^{2}$ on any compact subset of $(-\infty, 0]$;
\item[($ii$)] $\{\gamma_{n}\}_{n \in \N}$ uniformly converges to $\bar\gamma$ on every compact subset of $(-\infty, 0]$.
\end{itemize} 
\end{remarks}

Next we show that there exists a calibrated curve for Peierls barrier $\PB(x, \cdot)$ at any $x \in \R^{d}$.

\begin{proposition}\label{prop:viscositybarrier}
Let $x \in \R^{d}$. Then, for any $y \in \R^{d}$ there exists $(\bar\gamma, \bar{u}) \in \Gamma_{-\infty,0}^{\to y}$ calibrated for $\PB(x, \cdot)$ in $[-t, 0]$ for any $t > 0$, i.e., $(\bar\gamma, \bar{u})$ satisfy
\begin{equation}\label{eq:calibrata}
	\PB(x,y)-\PB(x,\bar\gamma(-t)) = \int_{-t}^{0}{L(\bar\gamma(s),\bar{u}(s))\ ds} - \mane t, \quad \forall\ t \geq 0.
\end{equation}
\end{proposition}
\proof

Fix $x$, $y \in \R^{d}$  and let $\{ t_{n}\}_{n \in \N}$, $(\gamma_{n}, u_{n}) \in \Gamma_{-t_{n}, 0}^{x \to y}$ be such that
\begin{equation*}
t_{n} \to \infty, \quad \text{and} \quad \lim_{n \to \infty} \int_{-t_{n}}^{0}{L(\gamma_{n}(s), u_{n}(s))\ ds} - \mane t_{n} = \PB(x,y).	
\end{equation*}
Then, by \Cref{rem:conv} there exists $(\bar\gamma, \bar{u})$ such that $u_{n}$ weakly converges to $\bar{u}$ and $\gamma_{n}$ uniformly converges to $\bar\gamma$, on every compact subset of $(-\infty,0]$.

 Fix $t \in [0,\infty)$, take $n \in \N$ such that $d_{n} = d_{\text{SR}}(\bar\gamma(-t), \gamma_{n}(-t)) \leq 1$ and  $t_{n} > t + 1$. Let $(\gamma_{0}, u_{0}) \in \Gamma_{-t, -t+d_{n}}^{\gamma_{n}(-t) \to \bar\gamma(-t)}$ be a geodesic pair and let $\widetilde{u}_{n} \in L^{2}(-t_{n}, -t + d_{n})$ be given by
\begin{align*}
\widetilde{u}_{n}(s) = 
\begin{cases}
	u_{n}(s), & \quad s \in [-t_{n}, -t]
	\\
	u_{0}(s), & \quad s \in (-t, -t +d_{n}].
\end{cases}	
\end{align*}
We denote by $\widetilde\gamma_{n}$ the associated trajectory, that is, $(\widetilde\gamma_{n}, \widetilde{u}_{n}) \in \Gamma_{-t_{n}, -t+d_{n}}^{x \to \bar\gamma(-t)}$. Then, defining the control $\widehat{u}_{n}(s)=\widetilde{u}_{n}(s-t_{n})$, so that $(\widehat\gamma_{n}, \widehat{u}_{n}) \in \Gamma_{0, t_{n}-t+d_{n}}^{x \to \bar\gamma(-t)}$, by {\bf (L2)} we get 
\begin{align*}
& A_{t_{n}-t + d_{n}}(x, \bar\gamma(-t)) - \mane(t_{n}-t+d_{n})
\\
\leq\ & \int_{0}^{t_{n} - t + d_{n}}{L(\widehat\gamma_{n}(s), \widehat{u}_{n}(s))\ ds} - \mane (t_{n}-t+d_{n})
\\
=\ & \int_{-t_{n}}^{-t+d_{n}}{L(\widetilde\gamma_{n}(s), \widetilde{u}_{n}(s))\ ds} - \mane (t_{n}-t+d_{n})
\\
=\ & \int_{-t_{n}}^{-t}{L(\gamma_{n}(s), u_{n}(s))\ ds} + \int_{-t}^{-t+d_{n}}{L(\gamma_{0}(s), u_{0}(s))\ ds} - \mane (t_{n}-t+d_{n}).
\end{align*}
Let us estimate 
\begin{equation*}
\int_{-t}^{-t+d_{n}}{L(\gamma_{0}(s), u_{0}(s))\ ds}.
\end{equation*}
Since $\| u_{0}\|_{\infty, [-t, -t+d_{n}]} \leq 1$, by \Cref{lem:boundedtrajectories} and \eqref{eq:compactnesstrajectory} we deduce that 
\begin{equation*}
|\gamma_{0}(s)| \leq (|\gamma_{n}(-t)| + c_{f}d_{n})e^{c_{f}d_{n}} \leq \Lambda(Q_{|x|}), \quad \forall\ s \in [-t, -t + d_{n}]
\end{equation*}
with $\Lambda(R) = (R+ c_{f})e^{c_{f}}$ for all $R\ge 0$. Thus, by {\bf (L2)} we obtain 
\begin{equation*}
\int_{-t}^{-t+d_{n}}{L(\gamma_{0}(s), u_{0}(s))\ ds} \leq 2d_{n}C_{1}(1 + \Lambda(Q_{|x|}))
\end{equation*}
and so
\begin{align*}
& A_{t_{n}-t + d_{n}}(x, \bar\gamma(-t)) - \mane(t_{n}-t+d_{n})
\\
\leq\ & \int_{-t_{n}}^{-t}{L(\gamma_{n}(s), u_{n}(s))\ ds} + (2C_{1}(1 + \Lambda(Q_{|x|}))-\mane)d_{n} - \mane(t_{n}-t).
\end{align*}
Hence, by the definition of Peierls barrier we have that 
\begin{equation*}
\PB(x, \bar\gamma(-t)) \leq \liminf_{n \to +\infty} \Big\{A_{t_{n}-t+d_{n}}(x, \bar\gamma(-t)) - \mane (t_{n}-t-d_{n})\Big\}.
\end{equation*}
Thanks to the lower-semicontinuity of the action, we obtain
\begin{equation*}
\int_{-t}^{0}{L(\bar\gamma(s), \bar{u}(s))\ ds} - \mane t \leq \liminf_{n \to +\infty}\left\{\int_{-t}^{0}{L(\gamma_{n}(s), u_{n}(s))\ ds} - \mane t \right\}.	
\end{equation*}
Therefore, combining the above estimates we get 
\begin{align*}
& \PB(x, \bar\gamma(-t)) + \int_{-t}^{0}{L(\bar\gamma(s), \bar{u}(s))\ ds} - \mane t 
\\
\leq\ & \liminf_{n \to +\infty} \Big\{A_{t_{n}-t+d_{n}}(x, \bar\gamma(-t)) - \mane (t_{n}-t-d_{n})\Big\} + \int_{-t}^{0}{L(\bar\gamma(s), \bar{u}(s))\ ds} - \mane t
\\
\leq\ & \liminf_{n \to +\infty}\left\{(2C_{1}(1 + \Lambda(Q_{|x|})) -\mane)d_{n} + \int_{-t_{n}}^{-t}{L(\gamma_{n}(s), u_{n}(s))\ ds} - \mane(t_{n}-t) \right\} 
\\
+\ & \liminf_{n \to +\infty}\left\{\int_{-t}^{0}{L(\gamma_{n}(s), u_{n}(s))\ ds} - \mane t \right\}.	
\end{align*}
By reordering the terms inside brackets we conclude that
\begin{align*}
& \PB(x, \bar\gamma(-t)) + \int_{-t}^{0}{L(\bar\gamma(s), \bar{u}(s))\ ds} - \mane t 
\\
\leq \ & \liminf_{n \to +\infty}\left\{(2C_{1}(1 + \Lambda(Q_{|x|})) - \mane)d_{n} + \int_{-t_{n}}^{0}{L(\gamma_{n}(s), u_{n}(s))\ ds} -\mane t_{n} \right\} = \PB(x,y).
\end{align*}
Therefore,
\begin{equation}\label{eq:call}
\PB(x, y) - \PB(x,\bar\gamma(-t)) \geq \int_{-t}^{0}{L(\bar\gamma(s), \bar{u}(s))\ ds} - \mane t	.
\end{equation}
Next, we claim that
\begin{equation}\label{eq:call1}
 \PB(x,y) - \PB(x, \bar\gamma(-t)) \leq \int_{-t}^{0}{L(\bar\gamma(s), \bar{u}(s))\ ds} - \mane t.
\end{equation}
Indeed, by \eqref{eq:trian} we have that 
\begin{equation*}
 \PB(x,y) - \PB(x, \bar\gamma(-t)) \leq A_{t}(\bar\gamma(-t), y) - \mane t.
\end{equation*}
Hence, defining the control 
\begin{equation*}
\widehat{u}(s)=\bar{u}(s-t), \quad s \geq 0
\end{equation*}
and denoting by $\widehat\gamma$ the associated trajectory, we deduce that 
\begin{align*}
\PB(x,y) - \PB(x, \bar\gamma(-t)) \leq\ & A_{t}(\bar\gamma(-t), y) - \mane t
 \\
\leq\ & \int_{0}^{t}{L(\widehat\gamma(s), \widehat{u}(s))\ ds} - \mane t =\ \int_{-t}^{0}{L(\bar\gamma(s), \bar{u}(s))\ ds} - \mane t. 
\end{align*}
By combining \eqref{eq:call} and \eqref{eq:call1} we obtain \eqref{eq:calibrata}. \qed


\begin{corollary}\label{cor:viscoHJ}
 For each $x \in \R^{d}$ we have that the map $y \mapsto \PB(x,y)$ is a critical solution of  \eqref{eq:HJcorrector} on $\R^{d}$. 
\end{corollary}
\proof
Note that $\overline\gamma$, constructed in \Cref{prop:viscositybarrier}, turns out to be calibrated for $\PB(x, \cdot)$ on any interval $[-t,0]$ with $t > 0$ and, thus, by \cref{eq:trian} we deduce that such a curve is also a minimizer for the action $A_{t}(x, \overline\gamma(-t))$. From \Cref{rem:peierlsubsolution} we have that the map $y \mapsto \PB(x,y)$ is a subsolution to the critical equation for any $x \in \R^{d}$, whereas in view of the fact that $\overline\gamma$ is a minimizing curve we also have that such a map is a supersolution for the critical equation. 

In order to show that $\PB(x, \cdot)$ satisfies \eqref{eq:corrector2}, let $z$ be in $\R^{d}$. Then, \eqref{eq:trian} yields
	\begin{equation*}
	\PB(x,z) \leq \min_{y \in \R^{d}} \{\PB(x,y) + A_{t}(y,z)- \mane t\}.
	\end{equation*}
	On the other hand, by \Cref{prop:viscositybarrier} we have that for any $x$, $z \in \R^{d}$ there exists $(\overline\gamma, \overline{u}) \in \Gamma_{-t,0}^{\to z}$ such that 
	\begin{equation*}
	\PB(x,z)-\PB(x,\bar\gamma(-t)) = \int_{-t}^{0}{L(\bar\gamma(s),\bar{u}(s))\ ds} - \mane t.
	\end{equation*}
	Consequently, there exists $y_{x,z} \in \R^{d}$ such that 
	\begin{equation*}
	\PB(x,z) = \PB(x,y_{x,z}) + A_{t}(y_{x,z},z) - \mane t. 
	\end{equation*}
	This shows that $\PB(x, \cdot)$ satisfies \eqref{eq:corrector2} for any $x \in \R^{d}$.  \qed


\subsection{Projected Aubry set}

In this section, we define the Aubry set and begin by showing some extra properties of  Peierls barrier on such a set. We conclude by proving that the projected Aubry set $\A$ is a compact subset of $\R^{d}$.

\begin{definition}[{\bf Projected Aubry set}]
		The projected Aubry set $\A$ is defined by
		\begin{equation*}
			\A = \big\{x \in \R^{d} : \PB(x,x)=0 \big\}. 
		\end{equation*}	
		\end{definition}

Next, we provide a characterization of the point in $\A$. 

\begin{proposition}\label{prop:aubrycal}
Let $x \in \R^{d}$. Then, the following properties are equivalent.
\begin{itemize}
\item[($i$)] $x \in \A$. 
\item[($ii$)] There exists a trajectory-control pair $(\gamma_{x}, u_{x}) \in \Gamma_{-\infty, 0}^{\to x} \cap \Gamma_{0,\infty}^{x \to}$ such that  
\begin{equation}\label{eq:rel1}
\PB(\gamma_{x}(t), x) = - \int_{0}^{t}{L(\gamma_{x}(s), u_{x}(s))\ ds} + \mane t	
\end{equation}
and
\begin{equation}\label{eq:rel2}
\PB(x, \gamma_{x}(-t)) = - \int_{-t}^{0}{L(\gamma_{x}(s), u_{x}(s))\ ds} + \mane t.	
\end{equation}
\end{itemize}
Moreover, the trajectory $\gamma_{x}$ in ($ii$) satisfies
\begin{equation}\label{eq:inclusion}
	\gamma_{x}(t) \in \A, \quad t \in \R.
	\end{equation}
\end{proposition}
\proof 

We start by proving that ($i$) implies ($ii$) and, in particular, we begin with \eqref{eq:rel1}. Since $x \in \A$, we have that $\PB(x,x)=0$ . So there exist $\{ t_{n}\}_{n \in \N}$ and $(\gamma^{+}_{n}, u^{+}_{n}) \in  \Gamma_{0,t_{n}}^{x \to x}$ such that
\begin{equation}\label{eq:zeros1}
t_{n} \to +\infty \quad \text{and} \quad \lim_{n \to +\infty} \int_{0}^{t_{n}}{L(\gamma^{+}_{n}(s), u^{+}_{n}(s))\ ds}-\mane t_{n} = 0. 	
\end{equation}
Then, by \Cref{lem:barrierconvergence} there exists $(\gamma^{+}_{x}, u^{+}_{x}) \in \Gamma_{0,\infty}^{x \to}$ such that $u^{+}_{n}$ weakly converges to $u^{+}_{x}$ and $\gamma^{+}_{n}$ uniformly converges to $\gamma^{+}_{x}$, on every compact subset of $[0,\infty)$, respectively. 
Fix $t \in [0,+\infty)$, fix $n$ large enough such that $d_{n} := d_{\SR}(\gamma^{+}_{x}(t), \gamma^{+}_{n}(t)) \leq 1$ and $t +1 < t_{n}$. Let $(\gamma_{0}, u_{0}) \in \Gamma_{t, t+ d_{n}}^{\gamma^{+}_{x}(t) \to\ \gamma^{+}_{n}(t)}$ be a geodesic pair and let $\widetilde{u}_{n} \in L^{2}(t, t_{n}+d_{n})$ be such that 
\begin{align*}
\widetilde{u}_{n}(s) = 
\begin{cases}
	u_{0}(s), & \quad s \in [t, t+d_{n}]
	\\
	u^{+}_{n}(s-d_{n}), & \quad s \in (t+d_{n}, t_{n}+d_{n}],
\end{cases}	
\end{align*}
so that, $(\widetilde\gamma_{n}, \widetilde{u}_{n}) \in \Gamma_{t, t_{n} + d_{n}}^{\gamma^{+}_{x}(t) \to x}$.  
Then, recalling that $\|u_0\|_{\infty, [t, t+d_{n}]} \leq 1$ by \Cref{lem:boundedtrajectories} and by \eqref{eq:compactnesstrajectory}, we get
\begin{equation*}
|\gamma_{0}(s)| \leq (|\gamma_{x}^{+}(t)| + c_{f}d_{n})e^{c_{f}d_{n}} \leq \Lambda(Q_{|x|}), \quad \forall\ s \in [t,t+ d_{n}]
\end{equation*}
with
\begin{equation*}
\Lambda(R) = (R + c_{f})e^{c_{f}} \quad(R\ge 0). 
\end{equation*} 
Thus, we obtain 
\begin{equation*}
\int_{t}^{t+d_{n}}{L(\gamma_{1}^{n}(s), u_{1}^{n}(s))\ ds} \leq 2d_{n}C_{1}(1+\Lambda(Q_{|x|})).
\end{equation*}
Hence,
\begin{align*}
\int_{t}^{t_{n}}{L(\widetilde\gamma_{n}(s), \widetilde{u}_{n}(s))\ ds} =\ &  \int_{t}^{t+d_{n}}{L(\gamma_{0}(s), u_{0}(s))\ ds} + \int_{t + d_{n}}^{t_{n}+d_{n}}{L(\gamma_{n}(s), u_{n}(s))\ ds} 
\\
\leq\ &   2C_{1}(1+Q_{|x|})) d_{n} + \int_{t + d_{n}}^{t_{n}+d_{n}}{L(\gamma_{n}(s), u_{n}(s))\ ds}.
\end{align*}
Now, defining $\widehat{u}_{n}(s)=\widetilde{u}_{n}(s+t)$, that is, $(\widehat\gamma_{n}, \widehat{u}_{n}) \in \Gamma_{0, t_{n}-t}^{ \gamma^{+}_{x}(t) \to x}$, we get
\begin{multline*}
 \PB(\gamma^{+}_{x}(t), x) \leq\  \liminf_{n \to +\infty} \big[A_{t_{n}+d_{n}-t} - \mane(t_{n} + d_{n} -t)\big]
\\
\leq\  \liminf_{n \to +\infty} \left[ \int_{0}^{t_{n} + d_{n} -t}{L(\widehat\gamma_{n}(s), \widehat{u}_{n}(s))\ ds} -\mane(t_{n} + d_{n} -t)\right]
\\
\leq\  \liminf_{n \to +\infty} \left[\int_{t}^{t+d_{n}}{L(\gamma_{0}(s), u_{0}(s))\ ds} + \int_{t + d_{n}}^{t_{n} + d_{n}}{L(\widetilde\gamma_{n}(s), \widetilde{u}_{n}(s))\ ds} - \mane (t_{n} + d_{n} -t)\right]
\\
\leq\   \liminf_{n \to +\infty} \left[ 2C_{1}(1+Q_{|x|})d_{n} + \int_{t + d_{n}}^{t_{n} + d_{n}}{L(\gamma^{+}_{n}(s-d_{n}), u^{+}_{n}(s-d_{n}))\ ds} - \mane (t_{n} + d_{n} -t)\right]
\\
=\  \liminf_{n \to +\infty} \Big[ \big(2C_{1}(1+Q_{|x|}) - \mane\big) d_{n} + \int_{0}^{t_{n}}{L(\gamma^{+}_{n}(s), u^{+}_{n}(s))\ ds}
\\
- \mane t_{n}  - \int_{0}^{t}{L(\gamma^{+}_{n}(s), u^{+}_{n}(s))\ ds} + \mane t \Big].	
\end{multline*}
Then, by \eqref{eq:zeros1}, the uniform convergence of $\gamma_{n}$, the fact that $d_{n} \downarrow 0$ and the lower semicontinuity of the functional we deduce that
\begin{align*}
\int_{0}^{t}{L(\gamma^{+}(s), u^{+}_{x}(s))\ ds} - \mane t + \PB(\gamma^{+}_{x}(t), x) \leq 0. 	
\end{align*}
Moreover, since $\PB(x, \cdot)$ is dominated by $L-\mane$ we also have that
\begin{align*}
\PB(x, \gamma^{+}_{x}(t)) = \PB(x, \gamma^{+}_{x}(t)) - \PB(x,x) \leq \int_{0}^{t}{L(\gamma^{+}_{n}(s), u^{+}_{n}(s))\ ds} - \mane t	
\end{align*}
and by \eqref{eq:peierls} we know that $\PB(x, \gamma^{+}_{x}(t)) + \PB(\gamma^{+}_{x}(t),x ) \geq 0$. Therefore, we obtain
\begin{equation*}
\int_{0}^{t}{L(\gamma^{+}_{x}(s), u^{+}_{x}(s))\ ds} - \mane t +\PB(\gamma^{+}_{x}(t), x) = 0. 	
\end{equation*}
Similar arguments show that there exists $(\gamma^{-}_{x}, u^{-}_{x}) \in \Gamma_{-\infty,0}^{\to x}$ such that \eqref{eq:rel2} holds. 
By taking the sum of \eqref{eq:rel1} and \eqref{eq:rel2} one immediately get that ($ii$) implies ($i$). 

Finally, observe that the inequality 
\begin{equation}\label{eq:claiminclusion}
\PB(\gamma_{x}(t), \gamma_{x}(t)) \leq 0 \quad (t \in\R)
\end{equation}
suffices to show \eqref{eq:inclusion}. Indeed, by \Cref{lemma} we have that $\PB(\gamma_{x}(t), \gamma_{x}(t)) \geq 0$. 
In view of \eqref{eq:peierls}, the following holds 
\begin{equation}\label{eq:triangular}
\PB(\gamma_{x}(t), \gamma_{x}(t)) \leq \PB(\gamma_{x}(t), x) + \PB(x, \gamma_{x}(t)) \quad (t \in \R). 
\end{equation}
Hence, combining \eqref{eq:rel1} and \eqref{eq:rel2} with \eqref{eq:triangular} we get \eqref{eq:claiminclusion}. \qed

We now show that the compact set $\K_{L}$ in {\bf (L3)}  intersects the $\omega$-limit set of any minimizer of the action.
\begin{proposition}
For any $R > 0$ there exists $T_{R} > 0$ such that for any $x \in \overline{B}_{R}$, any $t \geq T_{R}$ and any optimal pair $(\gamma_{x}, u_{x}) \in \Gamma_{0,t}^{x \to}$ for \eqref{eq:minimization} satisfies
	\begin{equation*}
		\LL^{1}\left(\{s \in [0,t]: \gamma_{x}(s) \in \K_{L} \} \right) > 0.
	\end{equation*}
\end{proposition}
\proof
We proceed by contradiction. Suppose that there exist $x_{0} \in \R^{d}$, $\{t_{k}\}_{k \in \N}$ with $t_{k} \to \infty$, and a sequence of optimal pairs $(\gamma_{k}, u_{k}) \in \Gamma_{0,t_{k}}^{x_{0} \to}$ of \eqref{eq:minimization} such that 
\begin{equation*}
		\LL^{1}\left(\{s \in [0,t_{k}]: \gamma_{k}(s) \in \K_{L} \} \right) = 0.
	\end{equation*}
On the one hand, we have that
\begin{equation}\label{eq:p11}
\int_{0}^{t_{k}}{L(\gamma_{k}(s), u_{k}(s))\ ds} \geq t_{k} \inf_{y \in \K_{L}^{c}} L(y,0).	
\end{equation}
On the other hand, recall that $\delta^{*}(x_{0})=d_{\SR}(x_{0}, x^{*})$ and let $(\gamma_{0}, u_{0}) \in \Gamma_{0, \delta(x_{0})}^{x_{0} \to x^{*}}$ be a geodesic pair. For any $k \in \N$ such that $t_{k} > \delta^{*}(x_{0})$ define the control
\begin{align*}
\widetilde{u}_{k}(s)=
\begin{cases}
	u_{0}(s), & \quad s \in [0,\delta^{*}(x_{0})]
	\\
	0, & \quad s \in (\delta^{*}(x_{0}), t_{k}].
\end{cases}	
\end{align*}
Since $\| u_{0}\|_{\infty, [0, \delta^{*}(x_{0})]} \leq 1$, by \Cref{lem:boundedtrajectories} we deduce that 
\begin{equation*}
|\gamma_{0}(s)| \leq (|x_{0}| + c_{f}\delta^{*}(x_{0}))e^{c_{f}\delta^{*}(x_{0})} \leq \Lambda(R), \quad \forall\ s \in [0, \delta^{*}(x_{0})]
\end{equation*}
with
\begin{equation*}
\Lambda(R)=\sup_{x \in \overline{B}_{R}} (|x| + c_{f}\delta^{*}(x))e^{c_{f}\delta^{*}(x)}.
\end{equation*}
Thus, by {\bf (L2)} we obtain 
\begin{equation*}
\int_{0}^{\delta^{*}(x_{0})}{L(\gamma_{0}(s), u_{0}(s))\ ds} \leq 2\delta^{*}(x_{0})C_{1}(1+\Lambda(R)).
\end{equation*}
Then, it follows that
\begin{align}\label{eq:p22}
\begin{split}
\int_{0}^{t_{k}}{L(\gamma_{k}(s), u_{k}(s))\ ds} \leq\ & \int_{0}^{t_{k}}{L(\widetilde\gamma_{k}(s), \widetilde{u}_{k}(s))\ ds} 
\\
\leq\ & 2\delta^{*}(x_{0})C_{1}(1+\Lambda(R)) + (t_{k}-\delta^{*}(x_{0}))L(x^{*},0).
\end{split}
\end{align}
Thus, combining \eqref{eq:p11} and \eqref{eq:p22} and dividing by $t_{k}$ we get
\begin{equation*}
	\inf_{y \in \K_{L}^{c}} L(y,0) < \frac{1}{t_{k}}\delta^{*}(x_{0})\ell_{1}(1+\Lambda(R)) + \left(1-\frac{\delta^{*}(x_{0})}{t_{k}} \right)L(x^{*},0)
\end{equation*}
Moreover, by {\bf (L3)} we deduce that
\begin{equation*}
	L(x^{*},0) + \delta_{L} < \frac{2}{t_{k}}\delta^{*}(x_{0})C_{1}(1+\Lambda(R))  + \left(1-\frac{\delta^{*}(x_{0})}{t_{k}} \right)L(x^{*},0)
\end{equation*}
with $\delta_{L} > 0$. Taking the limit as $k \to \infty$ in the above inequalities yields $\delta_{L} \leq 0$ which is a contradiction. \qed

\medskip

\begin{remarks}\label{rem:newrem}\em
In view of the reversibility of $L$, the above Lemma implies that for any $x \in \R^{d}$ there exists $T_{x} \geq 0$ such that for any $t \geq T_{x}$, and any optimal pair $(\gamma_{x}, u_{x}) \in \Gamma_{-t,0}^{\to x}$ for the problem 
\begin{equation}\label{eq:backproblem}
\inf_{(\gamma, u) \in \Gamma_{-t,0}^{\to x}} \int_{-t}^{0}{L(\gamma(s), u(s))\ ds}
\end{equation}
we have that
$
		\LL^{1}\left(\{s \in [-t,0]: \gamma_{x}(s) \in \K_{L} \} \right) > 0.
$
\end{remarks}

We observe that \Cref{rem:newrem} can be applied to calibrated curves for the action integral over $[-t,0]$ since such curves are, in particular, minimizing trajectories for \eqref{eq:backproblem}. This is a key point to deduce that the projected Aubry set is bounded, as we show below. 


\begin{proposition}\label{prop:boundedaubry}
$\A$ is a bounded subset of $\R^{d}$. 
\end{proposition}
\proof

Let $x_{0} \in \R^{d}$ be such that $\PB(x_{0}, x_{0}) = 0$. By \Cref{prop:viscositybarrier} there exists $(\bar\gamma, \bar{u}) \in \Gamma_{-\infty, 0}^{\to x_{0}}$ such that $\bar\gamma$ is a calibrated curve for $\PB(x_{0},\cdot\ )$. Moreover, by \Cref{rem:newrem}, we know that there exists $t_{0} \in (-\infty, 0]$ such that $\bar\gamma(t_{0}) \in \K_{L}.$ Thus, the trajectory, $\widetilde\gamma$, corresponding to  the control $$\widetilde{u}(s):=\bar{u}(s+t_{0}) \quad (s \in [0,-t_{0}])$$ with $\widetilde\gamma(-t_{0})=x_{0}$, is a calibrated curve for $\PB(x_{0}, \cdot\ )$ such that $\widetilde\gamma(0) \in \K_{L}$. Since, by \eqref{eq:compactnesstrajectory} the set of all points that be attained by action minimizers from a bounded set is itself bounded, we conclude that there exists $R_{L} \geq 0$ such that $x_{0} \in \overline{B}_{R_{L}}$. \qed

Next, we show that the projected Aubry set is closed. 
\begin{proposition}\label{prop:closedaubry}
$\A$ is a closed subset of $\R^{d}$.	
\end{proposition}
\proof
Let $\{x_{n}\}_{n \in \N}$ be a sequence in $\A$ such that $\displaystyle{\lim_{n \to \infty}} x_{n}=x \in \R^{d}$: we have to show that $x \in \A$. By definition we have that there exist subsequence, still labeled by $\{t_{n}\}_{n \in \N}$, and $\{(\gamma_{k_{n}}, u_{k_{n}})\}_{n \in \N} \in \Gamma_{0,t_{n}}^{x_{n} \to x_{n}}$ such that 
\begin{equation*}
\int_{0}^{t_{n}}{L(\gamma_{k_{n}}(s), u_{k_{n}}(s))\ ds} - \mane t_{n} \leq \frac{1}{n}.
\end{equation*}
Then, by \Cref{lem:barrierconvergence} there exists $(\bar\gamma, \bar{u})$ such that $u_{n}$ weakly converges to $\bar{u}$ and $\gamma_{n}$ uniformly converges to $\bar\gamma$, on every compact subset of $[0,\infty)$, respectively. Let us define $d_{n} = d_{\SR}(x_{n}, x)$ and the control 
\begin{align*}
\widetilde{u}_{n}(s)
\begin{cases}
	u^{n}_{1}(s), & \quad s \in [-d_{n}, 0]
	\\
	u_{n}(s), & \quad s \in (0, t_{n}]
	\\
	u^{n}_{2}(s), & \quad s \in (t_{n}, t_{n} + d_{n}]
\end{cases}	
\end{align*}
where $(\gamma^{n}_{1}, u^{n}_{1}) \in \Gamma_{-d_{n}, 0}^{x \to x_{n}}$ and $(\gamma^{n}_{2}, u^{n}_{2}) \in \Gamma_{t_{n}, t_{n} + d_{n}}^{x_{n} \to x}$ are geodesic pairs, on their respective intervals. Hence, we have that $(\widetilde\gamma_{n}, \widetilde{u}_{n}) \in \Gamma_{-d_{n}, t_{n} + d_{n}}^{x \to x}$.
Let us estimate $|\gamma_{1}^{n}(\cdot)|$ (the same estimate can be obtained for $|\gamma_{2}^{n}(\cdot)|$). Since $\|u^{n}_{1}\|_{\infty, [-d_{n}, 0]} \leq 1$, by \Cref{lem:boundedtrajectories} we get
\begin{equation*}
|\gamma_{1}^{n}(s)| \leq (|x| + c_{f}d_{n})e^{c_{f}d_{n}}\leq \Lambda(R), \quad \forall\ s \in [-d_{n}, 0]
\end{equation*}
where
\begin{equation*}
\Lambda(R) = (R + c_{f})e^{c_{f}} \quad(R\geq 0).
\end{equation*}
Thus, we obtain 
\begin{equation*}
\int_{-d_{n}}^{0}{L(\gamma_{1}^{n}(s), u_{1}^{n}(s))\ ds} \leq 2d_{n}C_{1}(1+\Lambda(|x|)).
\end{equation*}
Thus, we get
\begin{align*}
& \PB(x,x) \leq \liminf_{n \to \infty}\ [A_{t_{n} + 2d_{n}}(x,x) - \mane(t_{n} + 2d_{n})]
\\
\leq & \liminf_{n \to \infty} \Big( \int_{-d_{n}}^{0}{L(\gamma^{n}_{1}(s), u^{n}_{1}(s))\ ds} + \int_{0}^{t_{n}}{L(\gamma_{n}(s), u_{n}(s))\ ds} -\mane t_{n} 
\\
+ & \int_{t_{n}}^{t_{n} + d_{n}}{L(\gamma^{n}_{2}(s), u^{n}_{2}(s))\ ds} - 2\mane d_{n}\Big) \leq   \lim_{n \to \infty} \left(2d_{n}C_{1}(1+\Lambda(R)) + \frac{1}{n}\right)= 0.
\end{align*}
The proof is thus complete since, by definition, $\PB(x,x) \geq 0$ for any $x \in \R^{d}$. \qed

\begin{corollary}\label{thm:compactness}
	$\A$ is a nonempty compact set. 
\end{corollary}
\proof
From \Cref{prop:boundedaubry} and \Cref{prop:closedaubry} it follows that the Aubry set is compact. Moreover, notice that $\A$ is nonempty since $x^{*} \in \A$. \qed

	\section{Horizontal regularity of critical solutions}
	\label{sec:horizontal}
	
	\noindent{\it Throughout this section we assume that {\bf (F1)}, {\bf (F2)}, {\bf (L1)} -- {\bf (L3)} and {\bf (S)} are in force.}

%
%

	In this section we show that a critical solution is differentiable along the range of $F$ (see the definition below) at any point lying on the projected Aubry set.

\begin{definition}[{\bf Horizontal differentiability}]\label{def:horizontaldifferentiability}
We say that a continuous function $\psi$ on $\R^{d}$ is differentiable at $x \in \R^{d}$ along the range of $F(x)$ (or, horizontally differentiable at $x$) if there exists $q_{x} \in \R^{m}$ such that
\begin{equation}\label{eq:diffF}
\lim_{v \to 0} \frac{\psi(x + F(x)v) - \psi(x) - \langle q_{x}, v \rangle}{|v|} = 0.	
\end{equation}
\end{definition}

	Clearly, if $\psi$ is Frech\'et differentiable at $x$, than $\psi$ is differentiable along the range of $F(x)$ and $q_{x}=F^{\star}(x)D\psi(x)$. For any $\psi \in C(\R^{d})$ we set $D^{+}_{F}\psi(x)=F^{*}(x)D^{+}\psi(x)$.
	
	\begin{lemma}\label{lem:directional}
	Let $\psi \in C(\R^{d})$ be locally semiconcave. Then, $\psi$ is differentiable at $x \in \R^{d}$ along the range of $F(x)$ if and only if $D^{+}_{F}\psi(x) = \{ q_{x}\}$.   	
	\end{lemma}
\proof
We first prove that if $D^{+}_{F}\psi(x)$ is a singleton then $\psi$ is differentiable at $x$ along the range of $F(x)$. Let $\{q_{x}\}=D_{F}^{+}\psi(x)$ and take $p_{x} \in D^{+}\psi(x)$. Then
\begin{equation*}
\psi(x + F(x)v) - \psi(x) - \langle p_{x}, F(x)v \rangle \leq o(|F(x)v|) \leq o(|v|).	
\end{equation*}
Therefore, we deduce that
\begin{equation*}
\limsup_{v \to 0} \frac{\psi(x + F(x)v) - \psi(x) - \langle q_{x}, v\rangle}{|v|} \leq 0.	
\end{equation*}
In order to prove the reverse inequality for the $\liminf$, let $\{ v_{k}\}_{k \in \N}$ be any sequence such that $v_{k} \not = 0$, $v_{k} \to 0$ as $k \to +\infty$ and let
\begin{equation*}
p_{k} \in D^{+}\psi(x + F(x)v_{k}).	
\end{equation*}
Then
\begin{align*}
& \frac{1}{|v_{k}|}\big(\psi(x + F(x)v_{k}) - \psi(x) - \langle p_{x}, F(x)v_{k} \rangle \big) 
\\
=\ & \frac{1}{|v_{k}|}\big(\psi(x + F(x)v_{k}) - \psi(x) - \langle p_{k}, F(x)v_{k} \rangle + \langle p_{k} - p_{x}, F(x)v_{k} \rangle \big)	
\\
\geq\ & \frac{1}{|v_{k}|}o(|F(x)v_{k}|) - |F^{*}(x)p_{k} - q_{x}||v_{k}|.
\end{align*}
By the upper-semicontinuity of $D^{+}\psi$ we have that $|F^{*}(x)p_{k} - q_{x}| \to 0$ as $k \uparrow \infty$. Since since this is true for any sequence $v_{k} \to 0$, we conclude that
\begin{equation*}
\liminf_{ v \to 0}\frac{\psi(x + F(x)v)-\psi(x)-\langle q_{x}, v \rangle}{|v|} \geq 0.	
\end{equation*}

We now prove that, if $\psi$ is differentiable along the range of $F(x)$, then $D^{+}_{F}\psi(x)$ is a singleton. To do so, let $p \in D^{+}\psi(x)$ and let $q_{x} \in \R^{m}$ be as in \eqref{eq:diffF}. Then, we have that 
\begin{equation*}
	\lim_{h \downarrow 0} \frac{\psi(x+hF(x)\theta) - \psi(x)}{h} = \langle q_{x}, \theta \rangle.
\end{equation*}
Moreover, by definition we have that for any $\theta \in \R^{d}$
\begin{equation*}
\lim_{h \downarrow 0} \frac{\psi(x+hF(x)\theta) - \psi(x)}{h} \leq \langle F^{*}p, \theta \rangle.
\end{equation*}
Therefore, 
\begin{equation*}
	\langle q_{x}, \theta \rangle \leq \langle F^{*}p, \theta \rangle, \quad \forall\ \theta \in \R^{d}.
\end{equation*}
Thus $F^{*}(x)p=q_{x}$. 
\qed

\medskip
\noindent Hereafter, the vector $q_{x}$ given in  \Cref{def:horizontaldifferentiability} will be called the {\it horizontal differential} of $\psi$ at $x \in \R^{d}$ and will be denoted by $D_{F}\psi(x)$.

	The next two propositions ensure that any critical solution $\chi$ is differentiable along the range of $F$ at any point lying on a calibrated curve $\gamma$. The proof consists of showing that $D^{+}_{F}\chi$ is a singleton on $\gamma$.  We recall that 
	\begin{equation*}
	L^{*}(x,p)=\sup_{v \in \R^{d}} \big\{\langle p,v \rangle - L(x,v) \big\}	
	\end{equation*}
is the Legendre transform of $L$. We will rather write the critical equation using $L^{*}$, instead of the Hamiltonian $H$, to underline the role of horizontal gradients.

\begin{proposition}\label{prop:superdiff}
Let $\chi$ be a subsolution to the critical equation and let $(\gamma, u)$ be a trajectory-control pair such that $\gamma: \R \to \R^{d}$ is calibrated for $\chi$. Then we have that
\begin{equation*}
\mane + L^{*}(\gamma(\tau),p)=0, \quad \forall\ p \in D^{+}_{F}\chi(\gamma(\tau))
\end{equation*}
for all $\tau > 0$. 	
\end{proposition}
\proof
On the one hand, since $\chi$ is a subsolution of \eqref{eq:HJcorrector} we have that for any $\tau \geq 0$
\begin{equation*}
\mane + H(\gamma(\tau), p) \leq 0, \quad 	\forall\ p \in D^{+}\chi(\gamma(\tau)). 
\end{equation*}
So, recalling \eqref{eq:starhj} and the fact that $D^{+}_{F}\chi(x)=F^{*}(x)D^{+}\chi(x)$ for any $x \in \R^{d}$ we get
\begin{equation*}
\mane + L^{*}(\gamma(\tau), p) \leq 0, \quad 	\forall\ p \in D^{+}_{F}\chi(\gamma(\tau)) \,\, \forall \tau \geq 0.
\end{equation*} 
Thus, it is enough to prove the reverse inequality. 
	
	Let $\tau > 0$ and $\tau \geq h > 0$. Since $\gamma$ is a calibrated curve for $\chi$ we have that
	\begin{align*}
	& \chi(\gamma(\tau)) - \chi(\gamma(\tau-h)) = \int_{\tau-h}^{\tau}{L(\gamma(s), u(s))\ ds} - \mane h.
	\end{align*}
Then, by the definition of superdifferential we get
\begin{align*}
& \chi(\gamma(\tau)) - \chi(\gamma(\tau-h)) \leq\ \langle p, \gamma(\tau)-\gamma(\tau-h) \rangle + o(h)  
\\
=\ & \left\langle p, \int_{\tau-h}^{\tau}{\dot\gamma(s)\ ds} \right\rangle + o(h)  
 =\ \int_{\tau-h}^{\tau}{ \langle F^{*}(\gamma(s)) p, u(s) \rangle\ ds} + o(h)	
\end{align*}
Therefore,
\begin{align*}
	\int_{\tau-h}^{\tau}{L(\gamma(s), u(s))\ ds} - \mane h \leq \int_{\tau-h}^{\tau}{ \langle F^{*}(\gamma(s)) p, u(s) \rangle\ ds} + o(h)
\end{align*}
or
\begin{align*}
- \mane \leq\ & \frac{1}{h} \int_{\tau-h}^{\tau}{\Big(\langle F^{*}(\gamma(s))p, u(s) \rangle - L(\gamma(s), u(s))\Big)\ ds} + o(1)
\\
\leq\ & \frac{1}{h}\int_{\tau-h}^{\tau}{L^{*}(\gamma(s), F^{*}(\gamma(s))p)\ ds} + o(1).
\end{align*}
Thus, for $h \to 0$ we conclude that
\begin{equation*}
\mane + L^{*}(\gamma(\tau), F^{*}(\gamma(\tau))p) \geq 0. \eqno\square
\end{equation*}

\begin{proposition}\label{prop:differentiability}
Let $\chi$ be a critical solution and let $(\gamma, u)$ be a trajectory-control pair such that $\gamma: (a,b) \to \R^{d}$ is calibrated for $\chi$, where $-\infty\leq a<b\leq+\infty$. Then, for any $\tau \in(a,b)$ we have that $\chi$ is differentiable at $\gamma(\tau)$ along the range of $F(\gamma(\tau))$. 
\end{proposition}
\begin{proof}
We recall that, owing to \Cref{lem:semiconcavity}, $\chi$ is semiconcave. 
	Moreover, \Cref{prop:superdiff} yields
	\begin{equation*}
	\mane + L^{*}(\gamma(\tau), p) = 0, \quad \forall p \in D^{+}_{F}(\chi(\gamma(\tau)), \,\, \forall \tau \in(a,b).	
	\end{equation*}
Since $L^{*}(x, \cdot)$ is strictly convex and the set $D^{+}_{F}\chi(x)$ is convex,  the above equality implies that $D^{+}_{F}\chi(\gamma(\tau))$ is a singleton. Consequently, \Cref{lem:directional} ensures that  $\chi$ is differentiable at $\gamma(\tau)$ along the range of $F(\gamma(\tau))$. 
\end{proof}

We are now ready to prove the differentiability of any critical solution on the Aubry set.

		\begin{theorem}[{\bf Horizontal differentiability on the Aubry set}]\label{thm:aubryset}
		 Let $\chi$ be a critical solution. Then, the following holds. 
		\begin{itemize}
\item[($i$)] For any $x \in \A$ there exists $\gamma_{x}: \R \to \R^{d}$, calibrated for $\chi$, with $\gamma_{x}(0)=x$.
\item[($ii$)] $\chi$ is horizontally differentiable at every $x \in \A$.
		\end{itemize}
\end{theorem}

\begin{proof}
We begin by proving ($i$). Let $x \in \A$ and let $(\gamma_{x}, u_{x}) \in \Gamma_{-\infty,0}^{\to x} \cap \Gamma_{0,\infty}^{x \to}$ be the trajectory-control pair given by \Cref{prop:aubrycal}. We claim that $\gamma_{x}$ is calibrated for $\chi$. By \eqref{eq:corrector} we have that $\chi \prec L - \mane$ and, for any $t \geq 0$, the following holds
\begin{equation*}
\chi(\gamma_{x}(t)) - \chi(x) \leq \int_{0}^{t}{L(\gamma_{x}(s), u_{x}(s))\ ds} - \mane t.	
\end{equation*}
Moreover, in view of \eqref{eq:corrector2} we deduce that 
\begin{equation*}
	\chi(x) - \chi(\gamma_{x}(t)) \leq A_{s}(\gamma_{x}(t), x) - \mane s 
\end{equation*}
for any $s \geq 0$. Thus, since $\gamma_{x}$ is calibrated for $\PB(\cdot, x)$ we get
\begin{equation*}
\chi(x) - \chi(\gamma_{x}(t)) \leq \PB(\gamma_{x}(t), x) = - \int_{0}^{t}{L(\gamma_{x}(s), u_{x}(s))\ ds} + \mane t. 
\end{equation*}
This proves that $\gamma_{x}$ is a calibrated curve for $\chi$ on $[0,\infty)$. Similarly, one can prove that the same holds on $(-\infty, 0]$. Moreover, for any $s$, $t > 0$ we have that
\begin{multline*}
\chi(\gamma_{x}(t))- \chi(\gamma_{x}(-s)) =\ \chi(\gamma_{x}(t)) - \chi(x) + \chi(x) - \chi(\gamma_{x}(-s))
\\
=\  \int_{0}^{t}{L(\gamma_{x}(\tau), u_{x}(\tau))\ d\tau} - \mane t + \int_{-s}^{0}{L(\gamma_{x}(\tau), u_{x}(\tau))\ d\tau} - \mane s
\\
=\  \int_{-s}^{t}{L(\gamma_{x}(\tau), u_{x}(\tau))\ d\tau} - \mane (t+s),
 \end{multline*}
and this completes the proof of ($i$). 
Conclusion ($ii$) follows from  ($i$) and \Cref{prop:differentiability}. 
\end{proof}


\begin{corollary}\label{cor:feedback}
Let $\chi$ be a critical solution, let $x \in \A$, and let $\gamma_{x}$ be calibrated for $\chi$ on $\R$ with $\gamma_{x}(0)=x$. Then, $\gamma_{x}$ satisfies the state equation with control 
\begin{equation*}
u_{x}(t)=D_{p}L^{*}(\gamma_{x}(t), D_{F}\chi(\gamma_{x}(t))), \quad \forall t \in\R.
\end{equation*}
Moreover, 
\begin{equation*}
D_{F}\chi(\gamma_{x}(t))=D_{u}L(\gamma_{x}(t), u_{x}(t)), \quad \forall t \in \R. 
\end{equation*}
\end{corollary}
\proof
Let $\chi$ be a critical solution, let $x \in \A$ and let $\gamma_{x}$ be a calibrated curve for $\chi$. Let $u_{x}$ be the control associated with $\gamma_{x}$. Then, from the Maximum Principle and the inclusion of the dual arc into the superdifferential of the corresponding value function, e.g. \cite[Theorem 7.4.17]{bib:SC},  it follows that for a.e. $t\in\R$
\begin{equation*}
\langle D_{F}\chi(\gamma_{x}(t)), u_{x}(t) \rangle = L(\gamma_{x}(t), u_{x}(t)) + L^{*}(\gamma_{x}(t), D_{F}\chi(\gamma_{x}(t)).
\end{equation*}
Notice that, $D_{F}\chi(\gamma_{x}(t))$ exists by \Cref{prop:differentiability}.  Hence, by the properties of the Legendre transform and the fact that
$t\mapsto D_{p}L^{*}(\gamma_{x}(t), D_{F}\chi(\gamma_{x}(t)))$ is continuous on $\R$, 
we conclude that $u_{x}$ has a continuous extension to $\R$ given by 
$
u_{x}(t)=D_{p}L^{*}(\gamma_{x}(t), D_{F}\chi(\gamma_{x}(t)))
$.
\qed

\begin{remarks}\em
Following the classical Aubry-Mather theory for Tonelli Hamiltonian systems, one could  define the Aubry set $\widetilde\A \subset \R^{d} \times \R^{m}$ as 
\begin{equation*}
\widetilde\A = \bigcap\{(x,u) \in \A \times \R^{m}: D_{F}\chi(x)=D_{u}L(x,u)\}
\end{equation*}
where the intersection is taken over all critical solutions $\chi$. Note that such a set is nonempty since $(x^{*}, 0) \in \widetilde\A$. 
\end{remarks}

\bibliographystyle{abbrv}
\bibliography{references}

\end{document}